\documentclass[10pt]{amsart}
\usepackage[matrix,arrow]{xy}
\usepackage{amssymb,latexsym,amscd,amsmath}
\usepackage{pdflscape}

\input{xy}

\xyoption{all}

\numberwithin{equation}{section} 
\newtheorem{guess}{Theorem}[section]

\newtheorem{rem}[guess]{Remark}
\newtheorem{defi}[guess]{Definition}

\newtheorem{thm}[guess]{Theorem}

\newtheorem{lem}[guess]{Lemma}

\newtheorem{prop}[guess]{Proposition}

\newtheorem{Cor}[guess]{Corollary}

\newcommand{\cO}{\mathcal{O}}
\newcommand{\cC}{\mathcal{C}}
\newcommand{\cE}{\mathcal{E}}
\newcommand{\cR}{\mathcal{R}}

\newcommand{\cF}{\mathcal{F}}
\newcommand{\cH}{\mathcal{H}}
\newcommand{\cM}{\mathcal{M}}

\newcommand{\cG}{\mathcal{G}}
\newcommand{\cT}{\mathcal{T}}

\newcommand{\Quot}{\mathrm{Quot}}

\newcommand{\Sym}{\mathrm{Sym}}
\newcommand{\disc}{\mathrm{disc}}

\newcommand{\Aut}{\mathrm{Aut}}

\newcommand{\ram}{\mathrm{ram}}

\newcommand{\Img}{\mathrm{Img}}
\newcommand{\Spec}{\mathrm{Spec}}
\newcommand{\zero}{\mathrm{zero}}
\newcommand{\Ker}{\mathrm{Ker}}

\newcommand{\ev}{\mathrm{ev}}
\newcommand{\Tor}{\mathrm{Tor}}

\newcommand{\rank}{\mathrm{rank}}

\newcommand{\lra}{\longrightarrow}

\newcommand{\hra}{\hookrightarrow}

\newcommand{\ra}{\rightarrow}
\newcommand{\ol}{\overline}

\newcommand{\ms}{\mapsto}

\newcommand{\PP}{\mathbb{P}}

\newcommand{\ZZ}{\mathbb{Z}}
\newcommand{\GG}{\mathbb{G}}
\newcommand{\QQ}{\mathbb{Q}}

\newcommand{\HH}{\mathbb{H}}

\newcommand{\CC}{\mathbb{C}}

\newcommand{\Ext}{\mathrm{Ext}}

\newcommand{\Hom}{\mathrm{Hom}}

\newcommand{\Id}{\mathrm{Id}}

\newcommand{\GL}{\mathrm{GL}}
\newcommand{\SL}{\mathrm{SL}}
\newcommand{\SO}{\mathrm{SO}}

\newcommand{\Proj}{\mathrm{Proj}}

\newcommand{\Shalf}{S_{1/2}}

\begin{document}

\markboth{Yashonidhi Pandey}{}

%\catchline{}{}{}{}{}

\title{A properness result for degenerate Quadratic and Symplectic Bundles on a smooth projective curve}
\author{Yashonidhi Pandey}
\address{ 
Indian Institute of Science Education and Research, Mohali Knowledge city, Sector 81, SAS Nagar, Manauli PO 140306, India, ypandey@iisermohali.ac.in, yashonidhipandey@yahoo.co.uk}

\begin{abstract} Let $(V,q)$ be a vector bundle on a compact Riemann surface $X$ together with a quadratic form $q: \Sym^2(V) \ra \cO_X$ (respectively symplectic form $q: \Lambda^2V \ra \cO_X$). Fixing the degeneracy locus of the quadratic form induced on $V/\ker(q)$, we construct a coarse moduli of such objects. Further, we prove semi-stable reduction theorem for equivalence classes of such objects. In particular, the case when degeneracies of $q$ are higher than one is that of principal interest. We also provide a proof of properness of polystable orthogonal bundles without appealing to Bruhat-Tits theory in any characteristic.
\end{abstract}



\maketitle
\tableofcontents

\input{amssym.def}
%\newtheorem{guess}[section]{Question}
%\input{•} macros
% Calligraphische Zeichen

\section{Introduction}

\subsection{Motivation} 
We would like to construct a coarse moduli space for pairs $(V,q)$ where $V$ is a vector bundle of rank $n$ on a smooth projective curve $X$ and $q: \Sym^2V \ra \cO_X$ is a quadratic form on $V$ that is {\it only generically non-degenerate} as a {\it projective variety}. If $\deg(V)=0$  then $q$ must be a everywhere non-degenerate. This case is well understood. The case of $\deg(V) <0$ becomes that carrying interest. Note that if $q: \Sym^2 V \ra \cO_X$ factors through $\cO_X(-x) \hra \cO_X$ for some $x \in X$, then by going to a cover, we may take square root of $\cO_X(-x)$.  We have assumed in this article therefore that {\it no such factoring takes place}. 

Let $F$ be the functor that to a scheme $T$ associates equivalence classes of $T$-families of quadratic bundles $(V_T,q_T)$ on the curve $X$. The functor $F^{nd}$ associating generically non-degenerate quadratic forms is an open sub-functor of $F$. Let us fix the line bundle $\det(V^*)^2$, say $L$ and consider the associated functor $F^{nd}_L$. Then one has the natural map
\begin{equation} disc: F^{nd}_L \ra \PP(\Gamma(X,L))
\end{equation}
given by associating to a generically non-degenerate quadratic bundle $(V,q)$  its {\it discriminant section} $\det(q) \in \Gamma(X, \det(V^*)^2$.
Our interest has been in fixing the fiber of the map $disc$, which consists of fixing the degeneracy locus of $q$ and also the length of the sky-scraper sheaf $S=V^*/q(V)$ at each point in its support. The case when the depth of $S$ is more than one is our main interest (because depth one case can be studied equivalently as $\ZZ/2$-orthogonal bundles on two-sheeted covers). This does not seem to have been studied before. 

Let us fix $s \in \PP(\Gamma(X,\det(V^*)^2))$. On the fiber $disc^{-1}(s)$, the sky-scraper sheaf $V^*/q(V)$ defines a discrete invariant (Prop \ref{disceteinvariant}). The associated functor does not seem to be proper. On the limit, we obtain everywhere degenerate quadratic forms. To explain the limiting objects, let us consider quadratic bundles $(V,q)$ such that $\ker(q)$ has degree zero (it may possibly also have rank zero). We can extend the notion of {\it degeneracy locus} to such bundles by considering the degeneracy locus of the quadratic bundle $(V/\ker(q),\ol{q})$ where $\ol{q}$ denote the associated quadratic form. Putting $\ol{V}=V/\ker(q)$, we may also extend the notion of degeneracy type of $\ol{q}$ by considering the sky-scraper sheaf $\ol{V^*}/\ol{V}$. It is equivalently the torsion part of $V^*/V$.

In this paper, fixing the degeneracy locus and type of quadratic bundles in this extended sense, we have attempted to construct a coarse moduli space as a {\it projective scheme}. We show that the coarse moduli of pairs $(V,q)$, where $V$ is polystable of negative degree, {\it $q$ is generically non-degenerate} and $V^*/V$ is some fixed sky-scraper sheaf $S$, is compactified with everywhere degenerate quadratic bundles whose degeneracy locus is contained in $\zero(\det(q))$ and whose degeneracy type $\ol{V^*}/\ol{V}$ is a quotient of $S$ (cf Remark \ref{limitingobjects}). This implies that at prescribed points the order of degeneracies of a limiting object are atmost those of generic objects.

For the general case of vector bundles equipped with a quadratic form (possibly {\it everywhere degenerate}), a compact moduli space has been constructed in  \cite[Gomez-Sols]{gs}   using Geometric invariant theory for rank three. The authors find several non-linear conditions of semi-stability and expect more such conditions in the higher rank case. The key difficulty in arises in the application of the Hilbert-Mumford criterion when one tries to interpret the (semi)-stable points on a versal space in terms of the quotient object $(V,q)$ {\it intrinsically}. To check the semi-stability condition, in addition to sub-bundles, one needs filtrations  \cite[( Rem 2.3.3.2 ii), Page 144]{schmitt}). In \cite[Section 2]{opi}, the moduli of quadratic or symplectic bundles has been constructed for any rank using techniques of Differential geometry (cf also \cite{gpgi2} and \cite{gpm3} for symplectic case). The semi-stability conditions have also been explicited in \cite[Guidice-Pustetto]{lgp} for any rank  for quadratic bundles. In \cite[Definition 16, page 10]{lgp} they simplify the semi-stability condition on decorated principal bundles of A.Schmitt  (cf \cite[page 291 Theorem 2.8.1.2]{schmitt}), which should be checked for any weighted filtration, to those of length two.

\subsection{Statement of the main results}
Let $S$ be a sky-scraper sheaf on the curve $X$. 

Let  $\cF_\cC^S$ be the functor that to a scheme $T$ associates $T$-family of quadratic bundles $(V_T,q'_T)$ where $V_T$ is a $T$-family of {\it polystable vector bundles of fixed negative degree} such that the degeneracy type of $(V_t,q'_t)$ for $t \in T$, namely the torsion part of $V_t^*/V_t$ is a quotient sheaf of $S$ under obvious equivalence relation (cf subsection \ref{formulationfunctor}).

The main theorem of this paper is equivalent to proving 
\begin{thm} The functor $\cF_\cC^S$ is proper.
\end{thm}
Note here that the type of degeneracies of $q_T$ can be higher than one. This is the case of principal interest. 

When $S$ is trivial and $\ker(q_T)$ is always $0$, one obtains the case of orthogonal bundles. In section \ref{polystable}  we furnish an elementary proof of properness of polystable orthogonal bundles. This is the second main result of this paper. This proof does not appeal to Bruhat-Tits theory unlike results of Balaji-Seshadri \cite{bs}, Balaji-Parameshwaran \cite{parmu} and J.Heinloth \cite{heinlothssred}. It is also independnent of characteristic zero methods. Other known proofs in characteristic positive do require that induced vector bundles, like those obtained from low height representations, be also semi-stable. In contrast, for orthogonal bundles, our proof does not require any extra hypothesis. Lastly, in Remark \ref{diff} we mention the difference between the cases of orthogonal bundles and that of quadratic bundles with frames.

This main theorem is equivalent to another properness result which we state below after some preparation. We first begin by explaining the idea.

By restricting ourselves to generically non-degenerate quadratic bundles, we are able to reinterpret them as orthogonal bundles on two-sheeted cover $Y$ equipped with level structure together with a $\ZZ/2$-action.  More precisely, in Theorem \ref{eqcat} we show an equivalence of category between such objects and the category $\ZZ/2$-$O_n$ bundles $(W,q')$ on some Galois cover where the underlying vector bundle $W$ is moreover endowed with a {\it level structure} $f: W \ra \Shalf$ for some sky-scraper sheaf $\Shalf$ (cf. \cite[CSS]{cssbook} or  \cite[Huybrechts-Lehn]{hl} where such these are called {\it framed modules}). 

Similarly, this equivalence may be extended to an equivalence when $q$ is {\it everywhere degenerate}, between quadratic bundles $(V,q)$ of negative degree and  $\ZZ/2$-quadratic bundles $(W,q')$  with the framed structure $f:W \ra \Shalf$, where the underlying bundle $W$ {\it has degree zero} (see Section \ref{eqcatdeg}).

Using these reinterpretations, we construct in Theorem \ref{cms} a coarse moduli space for equivalence classes of semi-stable $\ZZ/2$-quadratic bundles with frames or equivalently quadratic bundles with prescribed degeneracies in {\it any rank}  for a certain extremal value of the parameter $\delta$. The construction is not entirely Geometric invariant theoretic. In general the moduli space exists only as a quasi-projective variety. We also describe the ``points of the moduli'', in other words, the so-called $S$-equivalence classes.

Now let us state the other properness result. Let $\cF_\cC^{\delta,\Shalf}$ be the functor which associates to a scheme $T$, equivalence classes of $T$-family of $\delta=1$ semi-stable $\ZZ/2$-quadratic bundles $(W,q')$ together with frame structure with values in $\Shalf$ where $W$ has degree {\it zero}. For the equivalence relation on semi-stable objects see subsubsection \ref{sequiclass}.

In Theorem \ref{ssred} we prove the following main theorem of this paper.

\begin{thm} For $\delta=1$, the functor $\cF_\cC^{\delta,\Shalf}$ is proper.
\end{thm} 

Thus under the above hypotheses, the quasi-projective moduli space of Theorem \ref{cms} is also proper.

%The strategy of proof follows \cite[BS]{bs}. In the course of the proof we were led to study vector bundles $(V,q)$ equipped with a quadratic form which is only generically non-degenerate. Such objects arise naturally whenever the degree of $V$ is strictly negative. Thereby we show that the local automorphism group $\Aut_x(V,q)$ can be realised as a subgroup of the local automorphism group of $\Aut_y(W,q')$ which is just the standard orthogonal group $O(q')(\cO_y)$. 

\subsection{Comments}
The functor that to a scheme $T$ associates semi-stable $T$-families of principal $G$-bundles on a curve $X$ does not seem to be proper (\cite[Introduction]{bs}). It seems necessary to pass to equivalence classes, which in this case consists of $S$-equivalence classes $G$-bundles. We state therefore the properness of {\it polystable} orthogogonal bundles in Section \ref{polystableorth} and {\it equivalence classes} of $\delta=1$ semi-stable quadratic framed modules in Section \ref{sssred}. Moreover in subsections \ref{extn-ul-fm} and \ref{extn-qf} we extend onto the special fiber, the framing morphism and the quadratic form that are fixed only {\it upto multiplication by scalars}.

We call the depth of a quadratic form $q$ to be the depth of the sky-scraper sheaf $V^*/q(V)$. Then in the {\it depth one} case, the level - can be alternatively recovered as the ``$-1$-eigenspace'' from the Galois action on the fibers. Thus in Theorem \ref{depth1}, we  deduce the properness of quadratic bundles in depth one case for any rank from the properness of the $\Gamma-\SO_n$-functor following the recent work say \cite{vbcss}. This provides an alternative compactification in the depth one case. In particular, in the case of quadratic bundles of rank $3$ as in \cite{gs}, if one tries to interpret them as parabolic vector bundles of rank $2$, then one recovers only the case depth one.

En route to the construction of moduli space, we needed to construct the moduli of framed modules \cite[Huybrechts-Lehn]{hl} (or Level structures \cite{css}) together with a Galois action. Since framed modules are known to pick up torsion in the limit, i.e a family of framed modules $f: W \ra \Shalf$ with $W$ locally free generically may contain an object in the limit where $W$ has torsion, so to be certain of the (semi)-stability conditions, we furnish proofs in the first version posted on arXiv. In this paper we content ourselves with just the statements because most proofs are straightforward generalizations of \cite{hl}. The semi-stability conditions appear with a parameter $\delta$. When $\delta > |\Shalf|$ then these spaces are empty. So we have often normalized $\delta$ by length $|\Shalf|$ of $\Shalf$.

We have used the equivalence, say when the quadratic form is generically non-degenerate, between vector bundles of negative degree and those of degree zero on some cover together with framed structure to draw out parallels between the general case of quadratic bundles and the well studied special case of orthogonal bundles. For instance to define the $S$-equivalence relation between semi-stable objects (cf Section \ref{pointsmoduli}) and the semi-stability conditions the parallels between the are easier to see. Most importantly, in the proofs of properness in Theorems \ref{ssred} and \ref{polystableorth} we see the crucial difference between them. Further as a tool they were  available \cite[HL]{hl}, so we decided to use them for the construction of coarse moduli space. 

The setting of \cite[HL]{hl} is for framed modules over a general projective scheme. Though our main application is over curves, but since the generalization of \cite[HL]{hl} to Galois situation follows easily, so we made the choice of casting results in that generality. 

We also show that our definition of (semi)-stability agrees with that in \cite{gs} for some suitable ranges in the values of the parameter to justify the sense in which we have generalised the result of \cite[GS]{gs}.

The preceding discussion also applies to pairs $(V,q)$ where $q: \Lambda^2 V \ra X$ is an alternating form that is only generically non-degenerate. To avoid repetition, we have indicated the changes only when it is crucial.
   
We have cross-checked our results when the framed structure is trivial (cf Remark \ref{cc}). 

\subsection{A Comparison} 
%with \cite[Gomez-Sols]{gs2}}
The setting of \cite{gs2} is over a smooth projective scheme. We recall Definition 5.1 \cite[]{gs2} which defines an orthogonal sheaf as a pair $(E,\varphi: \Sym^2 E \ra \cO_X)$ where $E$ is torsion free (OS3) and $\varphi_U: E_U \ra E^*_U$ induces an isomorphism over the open subset $U$ over which $E$ is locally free (OS4). Note that since over curves, the notions of torsion freeness and being locally free agree, so effectively $\varphi$ is an isomorphism over whole of $X$. In other words, $E$ is a usual orthogonal bundle. However, we are interested in $\deg(V)<0$ in which case $q$ can never induce an isomorphism.

\subsection{Layout}
The following three sections are preparatory intended mainly to define precisely the objects of interest. In subsection \ref{equi} we prove a somewhat formal result that allows us to reinterpret the problem of fixing degeneracy locus and type in terms of framed modules. In Section \ref{stenpro}, we list some important consequences of the extremal value of $\delta$ at one which will be used crucially. The section \ref{construction} exploits the idea of Ramanathan Lemma to construct the moduli by showing that the forgetful functor is affine. The sections \ref{sssred}, \ref{s-di} and \ref{polystableorth} form the heart of the paper proving the properness results. Finally we make a comparison with \cite[Gomez-Sols]{gs} and \cite[Guidice-Pustetto]{lgp} in the last section.

\subsection{Remark on Notation} 
 We identify a vector bundle $V$ with its sheaf of sections. The $2$-uple $(W,q')$ will always denote a vector bundle {\it of degree zero} together with a  quadratic form $q'$ on some Galois cover $Y$ and the notation $(V,q)$ will signify that $q$ is a quadratic form on the base curve $X$ and that the degree of $V$ is negative.

\section{Quadratic Bundles}
\begin{defi} A quadratic bundle $(V,q)$ is a pair where $q: V \ra V^*$ is a vector bundle homomorphism to the dual of $V$ such that  the composition of the natural isomorphism $V \ra (V^*)^*$ with the dual of $q$, namely $q^*: (V^*)^* \ra V^*$ identifies with $q$.
\end{defi}

When we will consider quadratic bundles $(W,q')$ of degree zero on some cover $Y$ of $X$, we will assume that the degree of $\ker(q)$ is also {\it zero}.

\begin{prop} \label{extsim} Let $(V,q)$ be a quadratic bundle on $X$ such that $q$ is generically an isomorphism. Then putting $S= V^*/q(V)$, there is a natural isomorphism of the sky-scraper sheaves $S \simeq \Ext^1_X(S, \cO_X)$.
\end{prop}
\begin{proof} This follows immediately by applying the functor $\Hom_X(-, \cO_X)$ to the short exact sequence 
$0 \ra V \stackrel{q}{\ra} V^* \ra S \ra 0 $
and remarking that $q^* =q$.
\end{proof}

\begin{defi} A morphism between quadratic bundles $\theta: (V_1,q_1) \ra (V_2,q_2)$ is an isomorphism $\theta: V_1 \ra V_2$ of the underlying vector bundles such that the following diagram commutes
\begin{equation*}
\xymatrix{
0 \ar[r] & V_1 \ar[r]^{q_1} \ar[d]^{\theta} & V_1^* \ar[r]                                       & S_1 \ar[r] & 0 \\
0 \ar[r] & V_2 \ar[r]^{q_2}                          & V_2^* \ar[u]^{\theta^*}   \ar[r]         & S_2  \ar[r]\ar@{-->}[u]_{\ol{\theta^*}} & 0
}
\end{equation*}
Similarly, one can define the local automorphism group $\Aut_x(V,q)$ of a quadratic bundle at $x$ by taking the stalks at a point. It is $g \in \Aut(V_x)$ such that
 \begin{equation} \label{autqbun}
\xymatrix{
0 \ar[r] & V_x \ar[r]^{q_1} \ar[d]^{g} & V_x^* \ar[r]                                       & S_x \ar[r] & 0 \\
0 \ar[r] & V_x \ar[r]^{q_2}                          & V_x^* \ar[u]^{g^*}   \ar[r]         & S_x  \ar[r]\ar@{-->}[u]_{\ol{g^*}} & 0
}
\end{equation}

\end{defi}

%For any quadratic bundle we have the following factorization $$ 0 \ra \ker(q) \ra V \stackrel{q}{\ra}  V/\ker(q) \stackrel{\overline{g}}{\dra}  (V/\ker(q))^* \ra V^*  \ra    \ker(q)^* \ra 0 $$

%We conjecture the following definition. \begin{defi} We say that a quadratic bundle $(V,q)$ is semi-stable if for every sub-bundle $W$ of $\ker(q)$, we have $$\mu(W \cap \ker(q) ) \leq \mu(\ker(q)) $$ and for every iso-tropic sub-bundle $U$ of $(V,q)$, we have $$ \mu (U ) \leq \mu(V/ \ker(q)).$$ We say that a quadratic bundle is stable if strict inequality holds in any of the above two conditions. \end{defi}

%Let $q: \Sym^2 V \ra \cO_X$ be a quadratic bundle, non-degenerate over the generic point of $X$. We shall further assume that the map $q$ doesn't factor through $\cO_X(-x) \ra \cO_X$ for any $x \in X$.

\begin{rem} The general case of quadratic bundles $q: \Sym^2 V \ra L$ with values in an arbitrary line bundle $L$ on $X$ can be reduced to this case by going to a $2$-sheeted covering $p:Y \ra X$ where $p^*L$ becomes the square of a line bundle $M$ on $Y$.  Then the moduli of $(p^*V, p^*q,p^*L)$ is isomorphic to $p^*q: \Sym^2p^*V \otimes M^{-1} \ra \cO_X$. 
\end{rem}

\section{Two equivalences of functors}
\subsection{First functor} \label{formulationfunctor}
We concentrate on quadratic bundles $(V,q)$ such that $q$ is generically an isomorphism. Let $n$ be a fixed integer that denotes the rank of the underlying vector bundles.

We now describe a datum that will be fixed throughout this paper. We fix 
\begin{enumerate}
\item a line bundle $L$ on $X$
\item a non-zero global section $s \in \Gamma(X,L)$
\item a sky-scraper sheaf $S$ on $X$ which is a quotient of $\cO_X^n$
\end{enumerate}
such that the section $s$ and the sheaf $S$ satisfy the following compatibility condition
$$div(s) = length(S).$$

Let $\cC$ denote the category whose objects are quadratic bundles $(V,q)$ satisfying the following properties
\begin{enumerate}
\item  $q: V \ra V^*$ is generically an isomorphism
\item  the quotient $V^*/q(V)$ is a sky-scraper sheaf denoted $S$ on $X$
\item $\det(V^*)^2=L$
\item the natural map $\det(q): \det(V) \ra \det(V^*)$ comes from the section $s \in \Gamma(X,L)$.
\end{enumerate}
Morphisms between such bundles are the usual morphisms between quadratic bundles.

Let $\cF_{\cC}: \{ Schemes\}^{op} \ra \{Sets \}$ be the functor which to a scheme $T$ associates the equivalence classes of $2$-uples $(V,q_T)$ where $V$ is a vector bundle on $X \times T$, $q_T: V \ra V^*$ is a symmetric morphism satisfying
 
\begin{enumerate} \label{fixprop}
\item $\det(V^*)^2 = p_X^*L \otimes p_T^*M_1$ for some line bundle $M_1$ on $T$
\item $V^*/q_T(V) = p_X^*S \otimes p_T^*M_2$ for some line bundle $M_2$ on $T$
\item For all $t \in T$, the associated map $q_t: V_t \ra V_t^*$ induces the section $s \in \Gamma(X,L)$ on taking determinants.
\end{enumerate} 

We say that two $2$-uples $(V_1,q_1)$ and $(V_2,q_2)$ are equivalent if

\begin{enumerate}
\item there exists a line bundle $N$ on $T$ 
\item a global section $\eta \in \Gamma(T,N^{-2})$ 
\item an isomorphism $\phi: V_1 \ra V_2 \otimes p_T^*N$
\end{enumerate}
such that the following diagram commutes
\begin{equation}
\xymatrix{
V_1 \ar[r]^{q_1} \ar[d]_{\phi} & V_1^* \\
V_2 \otimes p_T^*N \ar[r]_{q_2 \otimes p^*\eta} & V_2^* \otimes p_T^*N^* \ar[u]^{\phi^*}
}
\end{equation}

\subsection{Second Functor}
In this subsection we shall consider the following category $C$.

\begin{defi} Let $p:Y \ra X$ be a two-sheeted cover of $X$. Let $(W,q')$ be a vector bundle on $Y$ with an everywhere non-degenerate quadratic form $q'$, such that the action of the Galois group $\ZZ/2$ on $Y$ lifts to $(W,q')$. Suppose furthermore that we are given a natural surjective homomorphism $f: W \ra \Shalf$, where $\Shalf$ is a sky-scraper sheaf on $Y$ such that 
\begin{enumerate}
\item support of $\Shalf$ equals the ramification locus of $p$
\item length of $\Shalf$ equals the degree of the section $\disc$ of $p: Y \ra X$.
\item the induced action of $\ZZ/2$ on $\Ker(f)$ is trivial along the fibers.
\end{enumerate}

We call such a data a $\ZZ/2$-quadratic bundle on the curve $Y$ with frame structure $(f, \Shalf)$. Morphisms between such bundles are isomorphisms of the underlying bundles that respect the quadratic forms and frame structures.  Similarly, one can define local automorphism group $\Aut_y^{\ZZ/2}( (W,q') \ra \Shalf)$ by taking stalk at the point $y \in Y$. It is $g \in \Aut_y^{\ZZ/2}(W)$ such that
\begin{equation} \label{autobun}
\xymatrix{
0 \ar[r] & \Ker(q')_y \ar[r]  \ar@{-->}[d]^{g} & W_y^* \ar[r]                                       & {\Shalf}_y \ar[r] & 0 \\
0 \ar[r] & \Ker(q')_y \ar[r]                         & W_y^* \ar[u]^{g^*}   \ar[r]         & {\Shalf}_y  \ar[r]\ar[u]^{\ol{g^*}} & 0
}
\end{equation}

\end{defi}
We say that a $T$-family $f_T: W_T \ra p_Y^*\Shalf$ is {\it flat} if $f_t \neq 0 $ for all $t$ points of $T$.
Let $p: Y \ra X$ be a Galois cover of smooth projective curves with Galois group $\Gamma$. Let $\cF_C^{\Shalf}: \{ Schemes \}^{op} \ra \{ Sets \}$ be the functor that to a scheme $T$ associates a flat $T$-family $f_T: E=(W,q') \ra p_Y^*(\Shalf)$ of orthogonal bundles $E=(W,q')$ on $Y \times T$ along with a surjective morphism to $f_T: W \ra p_Y^*\Shalf$, such that for $t \in T$ the $\Gamma$-action of $p: Y \ra X$ lifts to $E_t$ and such that the induced $\Gamma$-action on $\Ker(W \stackrel{f_t}{\lra} \Shalf)$ is trivial. 

\subsection{Equivalence} \label{equi}
We now show the equivalence.
\begin{thm} \label{eqcat} 
The functors $\cF_\cC^S$ and $\cF_C^{\Shalf}$ are equivalent.
%We have an equivalence of categories between generically non-degenerate quadratic bundles $(V,q)$ on $X$ and $\ZZ/2$-quadratic bundles with frame structure over two-sheeted covers of $X$. Moreover for any scheme $T$, the above correspondence is respected by families of objects parametrized by $T$. 
\end{thm}
\begin{proof} Let $(V_T,q_T)$ be a $T$-family in $\cF_\cC$. The morphism  $\det(q_T) : \det(V_T) \ra \det(V^*_T)$ of line bundles on $X \times T$ gives a global section $s_T$
of $\det(V^*_T)^2$. Now by \ref{fixprop}(3), for any $t \in T$, the zero locus of $s_t$ is the same as $zero(s)$. Let $p: Y_s \ra X$ be a two-sheeted curve that realises the square root of this section. We describe this construction more formally to introduce the notation. 

We denote by $\det(V^*)$ by $L$ and $p: \Proj (\Sym^{\bullet} (\cO_X \oplus L^{-1})) \ra X$ by $\PP$. Let $x$ be the section of $\cO_{\PP}(1)$ defined by the "first co-ordinate" map 
$$ \cO_X \ra p_* \cO_\PP(1) = \cO_X \oplus L^{-1}$$
and $y$ be the section of $\cO_{\PP}(1) \otimes p^*(L)$ defined by the "second co-ordinate" map
$$ \cO_X \ra p_* (\cO_{\PP}(1) \otimes p^*(L)) = (\cO_X \oplus L^{-1}) \otimes L = L \oplus \cO_X. $$
Thus the section $y^2 - p^*(s) x^2$ is a section of $\cO_{\PP}(2) \otimes p^*L^2$. The curve $Y_s$ is defined as the zero sub-scheme of this section. We denote the projection onto $X$ again by $p$ and the zero locus of $s$ by $D$.

On $X \times T$, denoting $p_X: X \times T \ra X$, we have the following short exact-sequence that defines the sky-scraper sheaves
\begin{equation}
0 \ra V_T \stackrel{q_T}{\ra} V^*_T \ra p_X^*S \ra 0 
\end{equation}
\begin{equation}
0 \ra \det(V_T) \stackrel{s_T}{\ra} \det(V^*_T) \ra \cO_{div(s_T)} \ra 0. \label{detmap}
\end{equation}

For any $t \in T$, the following sets are equal : the degeneracy locus of $q_t$, support of the divisor $D$ , ramification locus $\ram(p)$ and the support of $S$.
 
%Pullling back \ref{detmap} to $Y$, we obtain \begin{equation} o \ra p^*(L^{-2}) \ra^{p^*(s)} \cO_Y \ra p^*(\cO_D) \ra 0. \end{equation}  Since $\zero(x) \cap \zero(y)$ is empty, so $\zero(x) \cap Y$ is empty. Thus the section $y$ on $Y$ defines the square root of $p^*(s)$ and defines \begin{eqnarray} 0 \ra p^*(L^{-1}) \ra^y  \cO_Y \ra \cO_{D/2}^q \ra 0, \\ 0 \ra p^*(L^{-2}) \ra^y  p^*(L^{-1}) \ra \cO_{D/2}^{\sub} \ra 0 \end{eqnarray} These quotients fit into the following short exact-sequence \begin{equation} 0 \ra \cO_{D/2}^{\sub} \ra p^*(\cO_D) \ra \cO_{D/2}^q \ra 0. \end{equation} We tensor the above sequence with $p^*(S)$.  The sheaf $p^*(\cO_D) \otimes p^*(S)$ is canonically isomorphic to $p^*(S)$. Moreover, since $ \cO_{D/2}^{\sub} \otimes p^*(S)$ surjects onto the kernel of $p^*(\cO_D) \otimes p^*(S) \ra \cO^q_{D/2} \otimes p^*(S)$, it follows by counting the length of the modules that we actually have an injection that gives the following short exact-sequence \begin{equation} 0 \ra \cO_{D/2}^{\sub} \otimes p^*(S) \ra p^*(S) \ra \cO_{D/2}^q \otimes p^*(S) \ra 0.  \end{equation} We denote $\cO_{D/2}^{\sub} \otimes p^*(S) $ by $\Shalf$.
\begin{defi} Multiplication by the section $y \in \cO_{\PP}(1) \otimes p^*L$ defines the morphism $m(y): p^*S \ra p^*S \otimes \cO_{\PP}(1) \otimes p^*L \simeq p^*S$, where for the last isomorphism, one uses local trivializations morphisms of $\cO_{\PP}(1) \otimes p^*L$ is small neighbourhoods around points in the support of $p^*S$. We define the "half" $\Shalf$ of the module $p^*(S)$ as the sub-module of $p^*S$ generated by the sub-modules $$\Img(m(y)^i) \cap \Ker(m(y)^i)$$ where $1 \leq i \leq \deg(s)$.
\end{defi}

The above definition shows the naturality of the definition of $\Shalf$.

%\begin{ex} Let $T$ be a local parameter at a ramification point in $y_1 \in \ram(p)$. Let us suppose that $p^*(S)$ identifies with $\CC[T]/T^2 \oplus \CC[T]/T^4 \oplus \CC[T]/T^6 \oplus \CC[T]/T^8$. Multiplication $m(y)$ by the section $y$ in the definition above corresponds to multiplication $m(T)$ by the local  parameter $T$. We have $$\begin{matrix} \Img(m(T)) \cap \ker(m(T)) & =  & T \CC[T]/T^2 & \oplus & T^3 \CC[T]/T^4 & \oplus & T^5 \CC[T]/T^6 & \oplus & T^7\CC[T]/T^8 \\ \Img (m(T^2)) \cap \ker((m(T^2)) & = & 0 & \oplus & T^2 \CC[T]/T^4 & \oplus & T^4 \CC[T]/T^6 & \oplus & T^6 \CC[T]/T^8 \\ \Img (m(T^3)) \cap \ker((m(T^3)) & = & 0 & \oplus & T^3 \CC[T]/T^4 & \oplus & T^3 \CC[T]/T^6 & \oplus & T^5 \CC[T]/T^8 \\ \Img (m(T^4)) \cap \ker((m(T^4)) & = & 0 & \oplus & 0 & \oplus & T^4 \CC[T]/T^6 & \oplus & T^4 \CC[T]/T^8 \\ \end{matrix}$$ And the sub-module generated by these sub-modules equals $$T \CC[T]/T^2 \oplus T^2 \CC[T]/T^4 \oplus T^3 \CC[T]/T^6 \oplus T^4 \CC[T]/T^8.$$ \end{ex}

\begin{prop}
The quotient module $p^*S/\Shalf$ is canonically isomorphic to $\Shalf$. So we obtain
\begin{equation}  \label{S/2seq}
 0 \ra \Shalf \ra p^*S \ra \Shalf \ra 0. 
 \end{equation}
\end{prop}
\begin{proof} The claim can be checked at stalks, and the stalks decompose as direct sums. The direct summands have even depth. So it suffices to check when $p^*(S)= \CC[T]/T^{2n}$. Now  $\Img(m(y)^i) \cap \Ker(m(y)^i)= T^{max \{i, 2n-i\}}\CC[T]/T^{2n}$. So the sub-module of $\CC[T]/T^{2n}$ generated, when $i$ varies in $1, \cdots, \deg(s)$, is $T^n \CC[T]/T^{2n}$. Now (\ref{S/2seq}) follows immediately.
\end{proof}
Let us recall the Rees lemma in Homological algebra.

\begin{thm}[Rees] \label{rees} Let $R$ be a ring and $x \in R$ be element which is neither a unit nor a zero divisor. Let $R^*= R/(x)$. For  a $R$-module $M$, suppose moreover that $x$ is regular on $M$. Then there is an isomorphism of 
$$Ext^n_{R^*} (L^*, M/xM) \simeq Ext^{n+1}_R(L^*, M)$$
for every $R^*$-module $L^*$ and every $n \geq 0$.
\end{thm}

\begin{lem} \label{ext1s} For any sky-scraper sheaf $S$ on $X$ we have $\Ext^1_X(S,\cO_X) \simeq S$.
\end{lem}
\begin{proof} Without loss of generality, we can suppose that the support of $S$ is just one point $p \in X$ and that $S= \cO_{X,p} /t^i$ where $t$ is a uniformizer at $p$. Applying Rees lemma to $t^i$, one gets $\Ext^1_X(\cO_{X,p}/t^i, \cO_X) \simeq \Hom_{\cO_{X,p}/t^i} (\cO_{X,p}/t^i, \cO_{X,p}/t^i) = \cO_{X,p}/t^i$.
\end{proof}

For convenience we denote $Y \times T \ra X \times T$ also by $p$. The extension class 
\begin{equation} \label{extup}
[0 \ra p^*V_T \ra p^*V^*_T \ra p^*_YS \ra 0 ] \in \Ext^1_{\cO_{Y \times T}} (p^*_YS, p^*V_T)
\end{equation}
 maps to an extension class in $\Ext^1_{\cO_{Y \times T}}(p_Y^*\Shalf, p^*V_T)$, which we denote by 
\begin{equation} \label{extW}
0 \ra p^*V \ra W \ra \Shalf \ra 0. 
\end{equation}

In what follows, for sake of readability we omit $T$, as if $T$ were we a point. The proof for an arbitrary $T$ carry over word for word. 

By ``transport of structure'', the action of the Galois group $\Aut(Y/X)$ on $p^*V$ extends to its Hecke-modification $W$.

\begin{prop} \label{fact-gen-non-deg} The quadratic form $p^*q$ on $p^*V$ extends uniquely to an everywhere non-degenerate quadratic form $q'$ on $W$.
\end{prop}
\begin{proof} Dualizing of the extension class (\ref{extW}), we get  
\begin{equation} \label{dualextW}
0 \ra W^* \ra p^*V^*\ra \Ext^1(\Shalf, \cO_Y) \ra  0
\end{equation}
We shall prove that the composite of $p^*V \stackrel{p^*q}{\lra} p^*V^*$ with $p^*V^* \ra \Ext^1(\Shalf, \cO_Y)$ is zero, so by (\ref{dualextW}), we would have a factorization 
\begin{equation} \label{start}
\xymatrix{
       & p^*V \ar[d]^{p^*q} \ar@{.>}[ld]^{q_1} & \\
W^* \ar[r] & p^*V^* \ar[r] & \Ext^1(\Shalf, \cO_Y)
}
\end{equation} 
Then we will show that composing  $q_1^*: W \ra p^*V^*$ with $p^*V^* \ra \Ext^1(\Shalf, \cO_Y)$ is zero. This furnishes the desired factorization
\begin{equation*}
\xymatrix{
      & W \ar[d]^{q_1^*} \ar@{.>}[ld]^{q'} & \\
W^* \ar[r] & p^*V^* \ar[r] & \Ext^1(\Shalf, \cO_Y).
}
\end{equation*}
Let us remark that $p^*V^* \ra \Ext^1(\Shalf, \cO_Y)$ corresponds to evaluating the sequence (\ref{extW}).

Firstly, since the extension (\ref{extup}) arises from  $\Id \in \Hom(p^*S, p^*S)$ under the connecting homomorphism, so in $\Ext^1(p^*S, p^*V^*)$ its image is zero. This corresponds to taking the push-out of (\ref{extup}) by $p^*q: p^*V \ra p^*V^*$. It follows now by the commuting squares,
 \begin{equation*}
\xymatrix{
\ar[r] &  \Hom(p^*S, p^*S) \ar[r] \ar[d] & \Ext^1(p^*S, p^*V) \ar[r] \ar[d] & \Ext^1(p^*S, p^*V^*) \ar[d] \\
\ar[r] &  \Hom(\Shalf, p^*S) \ar[r]      & \Ext^1(\Shalf, p^*V) \ar[r]      & \Ext^1(\Shalf, p^*V^*) 
}
\end{equation*}
the  push-out of (\ref{extW}) by $p^*q: p^*V \ra p^*V^*$ becomes the trivial extension
\begin{equation} \label{text}
\xymatrix{
0 \ar[r] & p^*V \ar[r] \ar[d]^{p^*q} & W \ar[r] \ar@{.>}[d] & \Shalf \ar[r] \ar@{.>}[d] & 0 \\
0 \ar[r] & p^*V^* \ar[r]                          & p^*V^* \oplus \Shalf \ar[r] & \Shalf \ar[r] & 0.
}
\end{equation}
The composite of $p^*V \stackrel{p^*q}{\ra} p^*V^* \ra \Ext^1(\Shalf, \cO_Y)$, corresponds to evaluating by  $v \in p^*V$ to get push out of  the bottom row of (\ref{text}). Taking the push-out of (\ref{text}) which is the trivial extension,
\begin{equation}
\xymatrix{
0 \ar[r] & p^*(V) \ar[r]  \ar@{->}[d]^{\ev(v)}            & p^*(V) \oplus \Shalf \ar[r] \ar@{.>}[d]^{\ev(v)} & \Shalf \ar[r] \ar@{.>}[d]& 0 \\
0 \ar[r] & \cO_Y \ar[r] & \cO_Y \oplus \Shalf \ar[r] & \Shalf \ar[r] & 0 
}
\end{equation}
we again get trivial extensions. This means, in other words, that the composite is zero. This furnishes the map $q_1 : p^*(V) \ra W^*$.  

By the sequence (\ref{dualextW}) and (\ref{extup}), we have a factorization in equation (\ref{start})
\begin{equation} 
\xymatrix{
       & p^*V \ar[d]^{p^*q} \ar[ld]^{q_1} & \\
W^* \ar[r] & p^*V^* \ar[r] \ar[d] & \Ext^1(\Shalf, \cO_Y) \\
            & p^*S \ar@{.>}[ru] &
}
\end{equation} 

 where $p^*S \ra \Ext^1(\Shalf, \cO_Y)$ is a natural surjective map . By Lemma \ref{ext1s}, we have $\Ext^1(\Shalf, \cO_Y) \simeq \Shalf$ and thus  $\ker( p^*S \ra \Ext^1(\Shalf, \cO_Y)) $ identifies with $\Shalf$ because $\Shalf$ is "half" of $p^*S$. We thus get
$$0 \ra p^*V \stackrel{q_1}{\ra} W^* \ra \Shalf \ra 0$$
by the diagram (\ref{start}).
Taking duals, we obtain
$$0 \ra W \stackrel{q_1^*}{\ra}  p^*V^* \ra \Ext^1(\Shalf, \cO_Y) \ra 0.$$
Thus the composite of $q_1^*$ with $p^*V^* \ra \Ext^1(\Shalf, \cO_Y)$ is zero. So we obtain a map $q': W \ra W^*$ factoring $p^*q: p^*V \ra p^*V^*$. Now $q'$ is an isomorphism: this follows immediately from the short exact sequence (\ref{S/2seq}) and the isomorphism $\Shalf \simeq \Ext^1(\Shalf, \cO_Y)$. 

%So we get the following diagram with exact rows \begin{equation} \xymatrix{  p^*(V)  \ar[d]^{q_1} \ar[r] \ar@/_/@<-3ex>[dd]_{p^*(q)} & W \ar[r] \ar@{.>}[d] & \Shalf  \ar[d] \\  W^* \ar[d]  \ar[r]                 & E_1 \ar[r]  \ar[d]      & \Shalf  \ar[d]  \\ p^*(V^*) \ar[r] & p^*(V^*) \oplus \Shalf \ar[r] & \Shalf }\end{equation} The composite of $ W \ra^{q_1^*} p^*(V^*) \ra \Ext^1(\Shalf, \cO_Y)$ corresponds to taking the push-out of the top sequence by $w \in W$ by evaluating through the middle sequence. So we need to show that for $w \in W$, the push-out of the middle sequence is trivial. For this we show that the middle row itself is split. The middle and the bottom rows correspond to $$0 \ra \Hom(\Shalf, W^*) \ra \Hom(\Shalf, p^*(V^*)) \ra\Hom(\Shalf, \Ext^1(\Shalf, \cO_Y)) \ra \Ext^1(\Shalf, W^*) \ra \Ext^1(\Shalf, p^*(V^*))$$ We take the push out of the middle row by $W^* \ra p^*(V^*)$ and then by $\Ext^1(\Shalf, \cO_Y)$ to get \begin{equation} \xymatrix{   W^* \ar[d]  \ar[r]                 & E_1 \ar[r]  \ar[d]      & \Shalf  \ar[d]  \\ p^*(V^*) \ar[r] \ar[d] & p^*(V^*) \oplus \Shalf \ar[r] \ar[d] & \Shalf \ar[d] \\ \Ext^1(\Shalf, \cO_Y)       \ar[r]            & \Ext^1(\Shalf,\cO_Y) \oplus \Shalf \ar[r]                    & \Shalf } \end{equation}The composition of $E_1 \ra p^*(V^*) \oplus \Shalf \ra p^*(V^*) \ra \Shalf$ is zero, since the vertical column is the sequence dual to \ref{extW}. It follows that the sequence of top row is split.

\end{proof}

%\begin{rem} Ofcourse the above proof is a co-ordinate free version of the following intuitive definition in local co-ordinates. We first define the quadratic form $q'$. For $y$ outside of $ \ram(p)$, the sheaf morphism $p^*(V) \ra W$ is an isomorphism at $y$. For $y \in \ram(p)$, let $(e_1,\cdots, e_n)$ be a $\cO_y$-basis of $W$ such that for some multiple of the uniformizing parameter $T$ of $\cO_y$, $(T^{m_1} e_1, \cdots, T^{m_n} e_n)$ form a basis of $p^*(V)$ and the quadratic form $p^*(q)$ is diagonal with respect to this basis. We define $$q'(e_i) = T^{-2 m_i} q(T^{m_i} e_i).$$ The following equalities show that $q'$ on $W$ extends the form $p^*(q)$ on $p^*(V)$ $$q'(T^{m_i} e_i) = T^{2 m_i} q'(e_i) = T^{2 m_i} (T^{- 2 m_i} q(T^{m_i} e_i) ) = q (T^{m_i} e_i).$$ Now we check the non-degeneracy of $q'$. Firstly, $$\Shalf = \oplus_{1 \leq i \leq n} \CC[[T]]/T^{m_i}$$ and the sequence $0 \ra \Shalf \ra p^*(S) \ra \Shalf \ra 0$ can be identified with $$0 \ra \oplus T^{m_i} \CC[[T]]/(T^{2 m_i}) \ra \oplus \CC[[T]]/(T^{2 m_i}) \ra \oplus \CC[[T]]/(T^{m_i}) \ra 0. $$ So the quadratic form $p^*(q)$ at $y$ is diagonal of the form $(T^{2 m_1}, \cdots, T^{2 m_n})$. This shows that $q'$ is regular and that in the basis $(e_1, \cdots, e_n)$ of $W$, it is given by the identity matrix that is non-degenerate. \end{rem}

Conversely, given $f_T: (W_T,q') \ra p_Y^*\Shalf$, where $p_Y: Y \times T \ra Y$, we take the $\Ker(f_T)$ and restrict the quadratic form to this sub-sheaf. Now the action of the Galois group becomes trivial on the fibers of $\Ker(f_T)$ and this bundle goes down to $X \times T$ along with the restricted quadratic form.

Now in the following proposition we assume that $T$ is a point i.e we check for objects in the category $\cC$. 
\begin{prop} \label{compunitgp} Let $y \in Y$ and $x = p(y)$. Then $\Aut_x(V,q) = \Aut_y^{\ZZ/2}((W,q') \ra \Shalf)$.
\end{prop}
\begin{proof} Firstly the ``half'' sub-module $S_{1/2}$ of $p^*(S)$ is canonically defined. So the inclusion $\Shalf \ra p^*(S)$ provides an inclusion at the stalks at $y$, so pulling back (\ref{autqbun}) to $Y$, we get
\begin{equation} \label{autlocW}
\xymatrix{
0 \ar[r] & p^*(V_x) \ar[r]  \ar[d]_{\pi^*(g)} & W_y^* \ar[r]                                       & {\Shalf}_y \ar[r] & 0 \\
0 \ar[r] & p^*(V_x) \ar[r]                         & W_y^* \ar[u]^{\pi^*(g^*)|_W}   \ar[r]         & {\Shalf}_y  \ar[r]\ar[u]_{\pi^*(\ol{g^*)|_{\Shalf}}} & 0
}
\end{equation}
This gives an element of $\Aut_y(W \ra \Shalf)$. After dualizing (\ref{autlocW}), we obtain the following diagram
\begin{equation}
\xymatrix{
\pi^*(V_x) \ar[r] \ar[d] & W \ar[r]^{q'}             & W^*   \ar[r] \ar[d] & \pi^*(V_x^*) \\
\pi^*(V_x) \ar[r]           & W \ar[r]^{q'} \ar[u]    & W^*  \ar[r]           & \pi^*(V_x^*) \ar[u]
}
\end{equation}
Since the left and the right and the extreme most squares commute, therefore so does the one in the middle as $W$ is a Hecke-modification of $\Gamma^*(V)$. This shows that $\pi^*(g) \in \Aut_y((W,q') \ra \Shalf)$. Its $\ZZ/2$-equivariance is formal.

Conversely, any $\Aut_y^{\ZZ/2}((W,q') \ra \Shalf)$ extends to an automorphism of 
$$ 0 \ra (\Ker(q'), q'|_{\Ker(q')}) \ra (W,q') \ra \Shalf \ra 0. $$
Now since the $\ZZ/2$-action on $W$ is compatible with $\Shalf$, so the restriction to $\Ker(q)$ is trivial along the fibers. Moreover since $g \in \Aut_y^{\ZZ/2}(W,q')$, so its restriction to $q'|_{\Ker}$ also goes down to give a quadratic bundle on $X$.
\end{proof}
It is a formal check that we have a bijection between isomorphisms between $V_1$ and $V_2$ that respect their quadratic structure and $\ZZ/2$-equivariant isomorphisms of $(W_1,q_1')$ and $(W_2, q_2')$ respecting their frame structures. This completes the proof of Theorem \ref{eqcat}.

\end{proof}

\subsection{Equivalence in case $q$ is everwhere degenerate} \label{eqcatdeg}
Now we consider the case of $(V,q)$ where $q$ is everwhere degenerate. We note that $q$ induces a quadratic form $\tilde{q}: V/\ker(q) \ra (V/\ker(q))^*$ which is generically non-degenerate. Note further that $coker(\tilde{q})$ identifies with the torsion part of $coker(q)$, which we denote as $S$ on $X$. As in the last section, by going to a covering, we may form the exact sequence 
\begin{equation} \label{tildeWext}
0 \ra p^*V/\ker(q) \ra \tilde{W} \ra \Shalf \ra 0
\end{equation} such that $\tilde{q}$ extends to $\tilde{q}'$ which is everywhere non-degenerate. 
\begin{prop} \label{degcase} The sequence $0 \ra \ker(q) \ra V \ra V/\ker(q) \ra 0$ embeds into a unique extension of $\tilde{W}$ by $p^*\ker(q)$. 
\end{prop}
\begin{proof}
Applying the functor $\Hom_Y(?,p^*\ker(q))$ to the sequence (\ref{tildeWext}) we get
\begin{eqnarray*}
\cdots \ra \Ext^1_Y(\Shalf,p^*\ker(q)) \ra \Ext^1_Y(\tilde{W},\ker(q)) \ra \\ \Ext^1_Y(p^*V/\ker(q), \ker(q)) \ra \Ext^2_Y(\Shalf, p^*\ker(q)) \ra \cdots
\end{eqnarray*}

But $\Ext^2$ vanishes since we are on a curve and the image of $\Ext^1$ is zero because it is torsion, being supported on the support of $\Shalf$, while $\Ext^1_Y(\tilde{W},\ker(q))$ is torsion free. Hence we obtain an isomorphism
\begin{equation}
\Ext^1_Y(\tilde{W},\ker(q)) \ra \Ext^1_Y(p^*V/\ker(q), \ker(q)),
\end{equation}
using which, for $0 \ra p^* \ker(q) \ra p^* V \ra p^* V/\ker(q) \ra 0$, we may find the unique extension $0 \ra p^* \ker(q) \ra W \ra \tilde{W} \ra 0$ of which it is the pull-back.
\end{proof}
We denote the middle term of the extension of $\tilde{W}$ suggestively as $W$.
As in the proof of Theorem \ref{eqcat}, the quadratic form $\tilde{q}$ extends to an everywhere non-degenerate quadratic form $\tilde{q}'$ on $\tilde{W}$. Then $\tilde{q}'$  may be used to endow $W$ with a quadratic form $q'$.

We may relax slightly the conditions for the category $\cC$ as follows.
\begin{enumerate}
\item  $q: V \ra V^*$ is just a {\it quadratic form}
\item  the {\it torsion part} of the quotient $V^*/q(V)$ is a sky-scraper sheaf denoted $S$ on $X$
\item $\det((V/\ker(q))^*)^2=L$
\item the natural map $\det(q): \det(V/\ker(q)) \ra \det((V/\ker(q))^*)$ comes from the section $s \in \Gamma(X,L)$.
\end{enumerate}

Similarly we may relax the conditions on the categroy $C$ by demanding only that $(W,q')$ is a quadractic bundle where $\deg(W)=0$ and $q'$ is simply a quadratic form.

Then from Prop \ref{degcase} and Theorem \ref{eqcat}, it follows that we will again have an isomorphism of the related functors. Further since $p^*\ker(q)=\ker(q')$, so we may further put the condition that $\deg(\ker(q))=0$ on objects $(V,q)$ of $\cC$ and $\deg(\ker(q'))=0$ on objects $(f:W \ra \Shalf,q')$ of $C$. Lastly note that the condition $(3)$ and $(4)$ of $\cC$ may further be relaxed, if we just want to fix the degeneracy locus and type, in the extended sense as in the introduction.

\section{Coarse moduli of $\Gamma$-Framed modules}
In this section (except the last subsection) we suppose more generally that $p: Y \ra X$ is a Galois cover of smooth projective schemes with Galois group $\Gamma$. Our purpose is only to state the generalizations of \cite{hl} to an equivariant setup. Most of the proofs are straightforward, so we omit them. 

Let $D$ be a $\Gamma$-invariant coherent $\cO_Y$-module and $\delta \in \QQ[X]$ a polynomial with positive leading coefficient.

\begin{defi} \label{framedmoduledef} A $\Gamma$-framed module is a pair consisting of a coherent $\cO_Y$-module $E$ invariant under $\Gamma$-action and a $\Gamma$-invariant homomorphism $\alpha: E \ra D$, called the framing of $E$. One calls $\ker(\alpha)$ as the kernel of the $\Gamma$-framed module $(E,\alpha)$. We put $\epsilon(\alpha)=1$ if $\alpha \neq 0$ and zero otherwise.  
\end{defi}

We shall denote $P_E(n) = \chi (E(n))$ the Hilbert polynomial of $E$ and by $P_{(E, \alpha)}= P_E - \epsilon(\alpha) \delta$ the Hilbert polynomial  of the pair $(E,\alpha)$. We put $h^0((E,\alpha)(m)) = h^0(E(m)) - \epsilon(\alpha) \delta(m)$.
 
Following \cite{hl}, we recall that if $E'$ is a coherent $\Gamma$-submodule of $E$ with quotient $E'' = E/E'$, then a framing $\alpha: E \ra D$ induces framings $\alpha'= \alpha|_{E'}$ and $\alpha''$ on $E''$ as follows: $\alpha''$ is zero if $\alpha' \neq 0$ and it is the induced homomorphism on $E''$ if $\alpha'$ is zero.

\begin{defi} A $\Gamma$-invariant homomorphism $\phi: (E,\alpha) \ra (E', \alpha')$ of $\Gamma$-framed modules is a $\Gamma$-homomorphism of underlying modules $\phi: E \ra E'$ for which there is an element $\lambda \in \CC$ such that $\alpha' \circ \phi_0 = \alpha$.
\end{defi}

\begin{defi} \label{mainss} A $\Gamma$-framed module $(E,\alpha)$ of rank $r$ is said to be (semi)stable with respect to $\delta$ with reduced Hilbert polynomial $p$, if $P_{(E,\alpha)}=r p$ and for all $\Gamma$-sub-modules $E'$ where $0 \neq E' \neq E$, of rank $r'$ induced framing $\alpha'$, the following inequality holds 
$$P_{(E', \alpha')} (\leq) r' p.$$
\end{defi}

\subsection{Semistable $\Gamma$-framed modules}
Given a flat $\Gamma$-family of framed modules $(E_T , \alpha_T: E_T \ra D_T)$, we can choose a fixed locally free coherent $\Gamma$-module $\tilde{D_T}$ and a surjective $\Gamma$-equivariant homomorphism $\phi: \tilde{D_T} \ra D_T$ with kernel $B$. Then following \cite{hl}, to $(E_T, \alpha_T)$ we associate $(\tilde{E_T}, \tilde{\alpha_T})$ by pulling back:
\begin{equation}
 \xymatrix{
          &                                       &  0  & 0 & \\
0 \ar[r]  & \ker(\alpha_T) \ar[r] \ar[d]^{\simeq} & E_T \ar[u] \ar[r]^{\alpha_T}                 & D_T \ar[r] \ar[u] & 0 \\
0 \ar[r]  & \ker(\alpha_T) \ar[r]                 & \tilde{E_T} \ar[r]^{\tilde{\alpha_T}} \ar[u] & \tilde{D_T} \ar[r] \ar[u] & 0 \\
          &                                       & B_T \ar[u] \ar[r]^{\simeq}                    & B_T \ar[u] &  \\
          &                                       & 0      \ar[u]                                  & 0 \ar[u]   &             
}
\end{equation}
The second row shows that $\tilde{E_T}$ is torsion free if the kernel of $\alpha$ is torsion free. This happens in particular if $(E,\alpha)$ is $\Gamma$-$\mu$-semistable.
\begin{lem} \label{1.11}
 If $(E, \alpha)$ is a $\Gamma$-framed module that can be deformed $\Gamma$-equivariantly to a framed module with torsion free kernel, then there is a morphism $\phi: (E, \alpha) \ra (G,\beta)$ of $\Gamma$-framed modules , such that
\begin{enumerate}
 \item $(G,\beta)$ has torsion free kernel.
\item $P_E=P_G$.
\item $P_{(E,\alpha)}=P_{(G,\beta)}$.
\item $\ker(\phi) = T(\ker(\phi))$.
\end{enumerate}

\end{lem}

\begin{lem} \label{1.6} If $(E, \alpha)$ and $(E', \alpha')$ are $\Gamma$-stable with the same reduced Hilbert polynomial $p$, then any non-trivial $\Gamma$-homomorphism $\phi: (E,\alpha) \ra (E',\alpha')$ is an isomorphism. Moreover we have $$\Hom((E,\alpha), (E',\alpha')) \simeq \CC.$$ If $\alpha \neq 0$, then there is a unique isomorphism $\phi_0$ with $\alpha' \circ \phi_0 = \alpha$.
\end{lem}

%\begin{proof} (standard) Suppose $\phi: (E, \alpha) \ra (E',\alpha')$ is a non-trivial $\Gamma$-equivariant homomorphism. The image $F= \img (\phi)$ inherits framings $\beta$ and $\beta'$ when considered as a quotient of $E$ and as a sub-module of $E'$, respectively. Let us remark that if $\beta' \neq 0 $ then $\beta \neq 0$ and $\beta' = \lambda \beta$ for some $\lambda \neq 0$. In any case, one has $$\rank(F) p \leq P_{(F, \beta)} \leq P_{(F', \beta')} \leq \rank(F) p$$ therefore equality holds in all places. Owing to $\Gamma$-stability, this implies that $E \simeq F \simeq E'$, that $\alpha = \beta \circ \phi$, that $\beta'=\alpha'$ and that $\beta$ and $\beta'$ differ by a non-trivial factor. Hence $\phi$ is an isomorphism of $\Gamma$-framed modules. For the remaining assertions, it suffices to show that $\Aut(E,\alpha) = \CC \Id_E$. Suppose that $\phi$ is an automorphism of $(E,\alpha)$. Let $x$ be any point in the support of $E$ and let $\mu$ be an eigenvalue of $\phi$ restricted to the fiber $E_x$. Then $\phi - \mu \Id_{E}$ is not surjective at $x$ and hence is not an isomorphism. So $\phi - \mu \Id_{E}$ must be zero by the first part of the proof. \end{proof}

The filtration in the following proposition is the `Jordan-Holder Filtration' in the case of $\Gamma$-framed modules.

\begin{prop} \label{jh} Let $(E,\alpha)$ be a $\Gamma$-semistable framed module with reduced Hilbert polynomial $p$. Then there is a filtration
$$E_\bullet : 0 = E_0 \subset E_1 \subset \cdots \subset E_s =E $$
such that all the factors $gr_i(E_\bullet) = E_i/E_{i-1}$ together with induced framings $\alpha_i$ are $\Gamma$-stable with respect to $\delta$ with reduced Hilbert polynomial $\delta$. The $\Gamma$-framed module $(gr(E), gr(\alpha)) = \oplus_i (gr_i(E),\alpha_i)$ is independent of the choice of the Jordan-Holder filtration upto permutation of factors.
\end{prop}

%\begin{proof} The proof is the standard as in the vector bundle case (cf \cite[Prop 1.13]{hl}) with the extra ingredient the following lemma: \begin{lem} \cite[Lemma 1.12]{hl} Let $F \subset G \subset E$ be a coherent module and $\alpha$ a framing of $E$. Then the framings induced on $G/F$ as a quotient of $G$ and as a submodule of $E/F$ agree. \end{lem} \end{proof}

\begin{defi} Two $\Gamma$-semistable framed modules $(E,\alpha)$ and $(E', \alpha')$ with reduced Hilbert polynomial $p$ are called $S$-equivalent, if their associated graded objects $(gr(E), gr(\alpha))$ and $(gr(E'), gr(\alpha'))$ are isomorphic.
\end{defi}

We remark that  $\Gamma$-stable framed module form $S$-equivalence class by themselves.

\begin{rem} Notice that at most one of the framings in $gr(\alpha)$ non-zero.
\end{rem}

\subsection{Equivalent conditions of $\Gamma$-semi-stability}
In this section we wish to state a technical results (Theorem \ref{2.1}) which logically corresponds to Theorem 2.1 of \cite[]{hl}.
The proofs in Section 2 Boundedness of \cite[]{hl} generalize in a straight forward way. So we only state the results.

Amongst the families of $\Gamma$-framed modules having torsion free kernel and satisfying the the following condition:
\begin{enumerate}
 \item we denote by $S^s$ those that are stable,
\item we denote by $S'_m$ those $(E, \alpha)$ satisfying $P(m) \leq h^0((E,\alpha)(m))$ and the inequality $h^0((E', \alpha')(m)) < r'/r P(m)$ for all $\Gamma$-sub-modules $(E',\alpha')$ of rank $r'$, where $0 \neq E' \neq E$.
\item we denote by $S''_m$ those $(E,\alpha)$ such that $h^0((E'', \alpha'')(m)) > r''/r P(m)$ for all quotient $\Gamma$-modules $(E'', \alpha'')$ of rank $r''$ where $E \neq E'' \neq 0$.
\end{enumerate}

The lemma below is an ingredient in proving Theorem \ref{2.1}.
\begin{lem} \label{2.4} There are integers $C$ and $m_1$ such that for all framed modules $(E,\alpha)$ in the family $S= S^s \cup \cup_{m \geq m_1} S''_m$ and for all saturated $\Gamma$-sub-modules  $(E', \alpha')$  the following holds:
$\deg (E') - r' \mu_P \leq C$ and either $-C \leq \deg(E') - r' \mu_P$  or
\begin{enumerate}
 \item $h^0((E',\alpha')(m)) < r'/r P(m)$ if $(E,\alpha) \in S^s$ and $m \geq m_1$
\item $r''/r P < P_{(E'',\alpha'')}$ if $(E,\alpha) \in S''_m$ for some $m \geq m_1$.
\end{enumerate}
Here $r'$ and $r''$ denote the ranks of $E'$ and $E''$ respectively. 
 \end{lem}
% \begin{proof} Let $c$ be the constant of the following lemma of Maruyama and Simpson. \begin{lem}  We fix a positive integer $r$. Then there exists a positive integer $c$ such that  for every $\mu$-semistable coherent $\cO_X$-module $F$ of positive rank strictly less than $r$ and slope $\mu$,  one has$$\frac{h^0(F)}{\rank(F)} \leq  \end{lem} \end{proof}

\begin{thm} \label{2.1} There is an integer $m_0$ such that the following properties of  $\Gamma$-framed modules having torsion free kernel with Hilbert polynomial $P$ are equivalent for $m \geq m_0$: 
\begin{enumerate}
 \item $(E,\alpha)$ is $\Gamma$-semistable.
\item  we have $P(m) \leq h^0((E,\alpha)(m))$ and $$h^0((E', \alpha')(m)) (\leq) r'/r P(m)$$ for all $\Gamma$-sub-modules $(E',\alpha')$ of rank $r'$, where $0 \neq E' \neq E$.
\item  $h^0((E'', \alpha'')(m)) (\geq) r''/r P(m)$ where $(E'', \alpha'')$ is a quotient $\Gamma$-module  of rank $r''$ and $E \neq E'' \neq 0$.
\end{enumerate}
 Moreover, for any framed module satisfying these conditions, $E$ is $m$-regular.
\end{thm}

As a corollary to the above theorem we have
\begin{lem} \label{2.9} If $(E,\alpha)$ is a semistable $\Gamma$-framed module, $m \geq m_0$ an integer, and $(E',\alpha')$ a submodule of rank $r'$ such that $h^0((E', \alpha')(m))=r'/r P(m)$, then $(E',\alpha')$ is semistable with reduced Hilbert polynomial $P/r$. 
\end{lem}

\subsection{Moduli of $\Gamma$-framed modules}
In this section we construct the moduli of semi-stable $\Gamma$-vector bundles with frames.
Let $V$ be a vector space of dimension $P_0(m)$. We denote $Hilb(V \otimes  \cO_Y(-m), P_0) $  by $\HH$,  $Grass(V \otimes H^0(Y, \cO_Y(l-m)), P_0(l))$  by $\GG_l$ and $\PP(\Lambda^{P_0(l)}(V \otimes H^0(\cO_Y(l-m))))$ by $\PP_l$. For sufficiently large $l$, we have closed immersions
$ \HH \ra \GG_l \ra \PP_l$.
Let $L$ denote the restricted ample line bundle on $\HH$ and let $\PP= \PP(\Hom(V, H^0_Y(D(m)))^*)$. Let $Z' \subset \HH \times \PP$ denote the closed sub-scheme of points $([q: V \otimes \cO_Y(-m) \ra F],[a: V \ra H^0_Y(D(m))])$ for which we have a factoring
\begin{equation}
 \xymatrix{
\cO_Y(-m) \otimes V \ar[r]^a \ar[d]^q & H^0_Y(D(m) \otimes \cO_Y(-m)) \\
F \ar@{.>}[ur] & 
}
\end{equation}
that induces the framing $\alpha: F \ra D$. The action of $\Gamma$ lifts to $Z'$. We let $Z'_\Gamma \subset Z'$ denote the sub-scheme fixed under $\Gamma$. The group $SAut_\Gamma(V)=\Aut_\Gamma(V) \cap \SL(V)$ acts diagonally on $Z'_\Gamma$ and the line bundles $\cO_{Z'}(n_1, n_2) = p_\HH^* L^{n_1} \otimes p_\PP^* \cO(n_2)$ carry natural $SAut_\Gamma(V)$ linearizations.

The following proposition is an analogue of Prop 3.1 in \cite[]{partha}.

\begin{prop} \label{ssinnum}
 For sufficiently large $l$, a point $([q],[a]) \in Z'_\Gamma$ is semi-stable   with respect to the linearization $\cO_{Z'_\Gamma}(n_1,n_2)$ if and only if
for any non-trivial proper $\Gamma$-linear subspace $V'$ of $V$, generating the subsheaf $F' \subset F$, we have 
$$\dim V' (n_1 P_0(l) + n_2) (\leq) \dim V (n_1 P_{F'}(l) + n_2 \epsilon (\alpha |_{F'})).$$
\end{prop}

We recall that $P_0$ is the polynomial $P-\delta$.

\begin{prop} \label{3.2} For sufficiently large $l$, a point $([q],[a]) \in Z'_\Gamma$ is semi-stable if and only if $(F, \alpha)$ is $\Gamma$-semi-stable and $q$ induces an isomorphism $V \ra H^0(Y,F(m))$.
\end{prop}

\begin{prop} \label{3.3} There exists a projective scheme $\cM^{ss}$ and a morphism $\pi: Z^{ss} \ra \cM^{ss}$ which is a good quotient for the $\SL(V)$ action on $Z^{ss}$. Moreover, there is an open subscheme $\cM^s \subset \cM^{ss}$ such that $Z^s = \pi^{-1}( \cM^s)$ and $\pi: Z^s \ra \cM^s$ is a geometric quotient. Two points $([q], [a]) $ and $([q'],[a'])$ are mapped to the same point in $\cM^{ss}$ if and only if the corresponding $\Gamma$-framed modules are $S$-equivalent.
\end{prop}

\begin{lem} \label{scl} Let $(E,\alpha)$ and $(F,\beta)$ be flat families of $\Gamma$-semi-stable framed modules  parametrized by a scheme $T$ of finite type over $\CC$ with the same reduced Hilbert polynomial $p$. Then the function
$$t \ms \dim_{k(t)} \Hom_\Gamma ((E_t, \alpha_t),(F_t, \beta_t)) $$
is upper semi-continuous in $t \in T$.
\end{lem}

\begin{thm} \label{modspace} Let $\delta \in \QQ[m]$ be a polynomial with positive leading coefficient and of degree strictly less than $\dim(X)$. There is a projective scheme $\cM_{\delta}^{ss}(Y \ra X, D, P)$ which is a coarse moduli space for the functor which associates to a scheme $T$ the set of isomorphism classes of flat families of semi-stable $\Gamma$-framed modules defined over $T$ with Hilbert polynomial $P$. Moreover, there is an open subscheme $\cM_{\delta}^s(Y \ra X, D, P)$ which represents the subfunctor of families of stable $\Gamma$-framed modules, i.e $\cM_{\delta}^s(Y \ra X, D, P)$ is a fine moduli space. A closed point in $\cM_{\delta}^{ss}(Y \ra X, D, P)$ represents an $S$-equivalence class of semi-stable $\Gamma$-framed modules.
\end{thm}

\section{Some remarkable properties of $\delta=1$ semi-stable framed modules on {\it curves}} \label{stenpro}
\begin{defi} \label{ssforfm} Let $f_1: V_1 \ra S_1$ and $f_2: V_2 \ra S_2$ be two framed modules. We define $(V_1,f_1, S_1) \otimes (V_2,f_2,S_2)$ as 
\begin{equation} \label{tensorprod}
 f_1 \otimes \Id_{V_2} + \Id_{V_1} \otimes f_2 : V_1 \otimes V_2 \ra S$$
where $S \hra S_1 \otimes V_2 \oplus V_1 \otimes S_2$ is defined as $$S =\{  a \oplus b \in S_1 \otimes V_2 \oplus V_1 \otimes S_2  | (\Id_{S_1} \otimes f_2)(a) = (f_1 \otimes \Id_{S_2})(b) \}.
\end{equation}
\end{defi}

\begin{prop} \label{kertor} If $(W,f)$ is semi-stable then $\ker(f)$ must be torsion free.
\end{prop}
\begin{proof}This follows by applying the definition of semi-stability of framed modules to the submodule $\ker(f) \cap Torsion(W) \subset W$.
\end{proof}
\begin{prop} \label{delta=1surj} If $f: W \ra S_{1/2}$ is a $\delta=|S_{1/2}|$-semi-stable framed module then the framing $f$ must be surjective.
\end{prop}
\begin{proof}
This follows by observing that $\deg(\ker(f))=\deg(W)-|\Img(f)|$, and applying the semi-stability condition.
\end{proof}
\begin{prop} \label{kerstable} A framed module $(V,f, S)$, where $f: V \ra S$ is surjective, is $\delta=|S|$-(semi)stable if and only if $\ker(f)$ is a (semi)stable vector bundle in the usual sense.
\end{prop}
\begin{proof} Suppose that $(V,f)$ is $\delta$-(semi)stable. Notice that $$\deg(\ker(f)) = \deg(V) - \chi(V) \delta.$$ Now for any subbundle $W \subset \ker(f)$ we have $\chi(W) =0$. So $\ker(f)$ is (semi)stable follows by Definition \ref{mainss}.

Conversely suppose that $\ker(f)$ is semi-stable. Then for any subsheaf $W$ of $V$, let $W_1 = \ker(W \ra S)$. Now $V$ is $\delta$-(semi)stable follows from 
\begin{eqnarray*}
\frac{\deg(W) - \chi(W) \delta}{\rank(W)} \leq \frac{\deg(W_1) - \chi(W_1) \delta}{\rank(W_1)} = \frac{ \deg(W_1)}{\rank(W_1)} \\ (\leq) \frac{\deg(\ker(f))}{\rank(\ker(f)} = \frac{\deg(V) - \chi(V) \delta }{\rank(V)}.
\end{eqnarray*}
\end{proof}

\begin{rem} Since a vector bundle is $\Gamma$-semi-stable if and only if it is semi-stable, hence we get that a $\Gamma$-framed module over a curve is $\Gamma$-semi-stable if and only if the it is semi-stable for $\delta=1$.
\end{rem}
\begin{prop} \label{tenprod} Let $f_1: V_1 \ra S_1$ and $f_2: V_2 \ra S_2$ are $\delta=|S_1|$ (resp $|S_2|$) (semi)-stable framed modules where $f_1$ and $f_2$ are surjective. Then $(V_1,f_1,S_1) \otimes (V_2,f_2,S_2)$ is $ |S|$ (semi)-stable (here $S$ is as in Definition \ref{ssforfm}).
\end{prop}
\begin{proof}
Firstly, it follows that $\ker(f_1)=K_1$ and $\ker(f_2)=K_2$ are (semi)-stable vector bundles by Propositions \ref{kertor} and \ref{kerstable}.
Now we show that 
\begin{equation} \label{checktenprod}
\ker(f_1 \otimes \Id_{V_2} + \Id_{V_1} \otimes f_2)=K_1 \otimes K_2. 
\end{equation} The inclusion $\subset$ is clear. For the other inclusion, consider the following diagram and let $v \in V_1 \otimes V_2 $ be in the kernel. Consider the following diagram where all the horizontal arrows on the right are surjective.
%\begin{landscape}
\begin{tiny}
\begin{equation*}
\xymatrix{
 & & & 0=Tor_1(V_1,K_2) \ar@{>>}[r] \ar[d] & 0=Tor_1(S_1,K_2)  \ar[d] \\
   & & &Tor_1(V_1,V_2) \ar@{>>}[r] \ar[d] & Tor_1(S_1,V_2)  \ar[d]  \\
  & & Tor_1(K_1,S_2)=0 \ar[d] \ar[r] & Tor_1(V_1,S_2) \ar@{>>}[r] \ar[d] & Tor_1(S_1,S_2)  \ar[d]  \\
  0 =Tor_1(V_1,K_2) \ar[r] \ar[d] & 0= Tor_1(S_1,K_2) \ar[d] \ar[r] & K_1 \otimes K_2 \ar[d] \ar[r] & V_1 \otimes K_2 \ar@{>>}[r] \ar[d] & S_1 \otimes K_2  \ar[d]  \\
  Tor_1(V_1,V_2) \ar[r] \ar@{>>}[d] & Tor_1(S_1,V_2) \ar[r] \ar@{>>}[d] & K_1 \otimes V_2 \ar[r] \ar@{>>}[d] &  V_1 \otimes V_2 \ar@{>>}[r] \ar@{>>}[d] & S_1 \otimes V_2  \ar@{>>}[d]  \\
 Tor_1(V_1,S_2) \ar[r]  & Tor_1(S_1,S_2) \ar[r]  & K_1 \otimes S_2 \ar[r]  & V_1 \otimes S_2 \ar@{>>}[r]  & S_1 \otimes S_2  
}
\end{equation*}
\end{tiny}
%\end{landscape}

Then $v \in V_1 \otimes V_2$ comes from an element $v_1$ in $K_1 \otimes V_2$. The image of $v_1$ in $K_1 \otimes S_2$ comes from an element $v_2 \in Tor_1(S_1, S_2)$. Now since $K_1$ and $K_2$ are torsion free by Proposition \ref{kertor}, so the morphism $Tor(S_1,V_2) \ra Tor_1(S_1,S_2)$ is an isomorphism. Now one can lift $v_2$ to $Tor_1(S_1,V_2)$ and use it to  alter $v_1$ so that $v_1$ still lifts $v$ but its image in $K_1 \otimes S_2$ is zero i.e $v_1$ comes from an element in $K_1 \otimes K_2$ as desired.

By the tensor product theorem for vector bundles, it follows that $K_1 \otimes K_2$ is (semi)-stable. Now by Proposition \ref{kerstable}, it follows that $(V_1,f_1,S_1) \otimes (V_2,f_2,S_2)$ is (semi)-stable.
\end{proof}

\begin{defi} \label{2ndsym} Let $f: W \ra S_W$ be a framed module. Following notation of \ref{tensorprod}, we define its {\it second symmetric product} as $f_1: \Sym^2 W \ra S_1$ where $f_1$ is the restriction of $f \otimes \Id_W \oplus \Id_W \otimes f$ to $\Sym^2W \hra W \otimes W$ and $S_1 \subset S$ consisting of fixed points of the involution $i$ on $S$ induced by the exchange of $W \otimes S$ with $S \otimes W$. More precisely on $S_W \otimes W \oplus W \otimes S_W$ consider the involution $I$ defined on indecomposable  tensors as $s_1 \otimes w_1 + w_2 \otimes s_2 \ms s_2 \otimes w_2 + w_1 \otimes s_1$. We have $I_S =i$. Similarly we define the {\it second wedge product} as $f_\Lambda: \Lambda^2 W \ra S_\Lambda$ where $S_\Lambda \subset S$ consists of the $-1$-eigenspace of $i$-action on $S$. 
\end{defi} 

\begin{prop} \label{2ndsymresult} If $f: W \ra S$ is a $\delta=1$ semi-stable framed module, then in the notation of Definition \ref{2ndsym}, $f_1: \Sym^2W \ra S_1$ is also $\delta=1$ semi-stable. Further $\ker(f_1)= \Sym^2 \ker(f)$. Similarly $f_\Lambda: \Lambda^2 W \ra S_\Lambda$ is also $\delta=1$-semi-stable and $\ker(f_\Lambda)=\Lambda^2 \ker(f_\Lambda)$.
\end{prop}
\begin{proof} It is easy to see that the framing map $f_1$ is actually surjective. We can further check that if $V=\ker(f)$ then $\Sym^2V= \ker(f_1)$. Now since $\delta=1$, so $V$ is a semi-stable vector bundle. Thus so is $\Sym^2V$. This implies that for  $\delta=1$, $f_1: \Sym^2W \ra S_1$ is semi-stable by Proposition \ref{kerstable}. The last assertion follows from (\ref{checktenprod}).
\end{proof}

\section{Construction of coarse moduli space of $\Gamma$-quadratic bundles with frames} \label{construction}
In this section, it will be convenient to normalize $\delta$ by dividing it by the length of the sky-scraper sheaf $|S|$. Thus the moduli is empty if $\delta>1$.

\begin{rem} With $\delta=1$, for any framed module $f: W \ra \Shalf$, $\ker(f)$ is always destabilizing. But in checking {\it stability} we shall ignore this sub-sheaf.
\end{rem}

Now we specialize $p: Y \ra X$ to be a Galois cover of smooth projective {\it curves}.

\begin{defi} \label{ssdeforfs} Let $(W,q')$ be a $\Gamma$-quadratic bundle on $Y$ with a surjective homomorphism $f: W \ra \Shalf$ to a sky-scraper sheaf $\Shalf$. We shall say that $(W,q',f)$ is $\delta$-semi-stable if the underlying framed module $f: W \ra \Shalf$ is $\delta$-semi-stable as a framed module.

%For a sub-bundle $W_1$ of $W$, we denote by $\epsilon(W_1)$ the length of the image of $W_1$ in $\Shalf$. We say that $(W,q') \ra \Shalf$ is $\delta$-semi-stable if for any  $\ZZ/2$-sub-bundle $W_1$, we have$$\frac{\deg(W_1)}{\rank{(W_1)}} - \delta (\frac{\epsilon(W_1)}{\rank(W_1)} - \frac{\epsilon(W)}{\rank(W)})  \leq 0.$$
\end{defi}

\begin{rem} This definition of (semi)-stability above is a direct generalization of the definition stated in Huybrechts-Lehn \cite[HL]{hl} and Seshadri \cite[CSS]{css}. 
%In \cite[HL]{hl}, a vector bundle $(V, f: V \ra S)$ with frame structure is defined to be (semi)-stable if for every sub-bundle $V'$ of $V$, we have \begin{equation} \frac{\deg(V') - \chi(V') |S|}{\rank(V')} \leq \frac{\deg(V) - \chi(V)|S|}{\rank(V)} \end{equation} where $\chi(V')=0$ if $\Img(f(V')=0$ and $1$ otherwise. Applying the above definition to the kernel of $f$ restricted to $V'$ we get our definition in the vector bundle case. 
\end{rem}

\begin{Cor} For $\delta=1$, if $|S|$ and $\rank(W)$ are coprime then semi-stable points are stable too.
\end{Cor}

In this section we are interested in constructing the coarse moduli space for the functor $\cF_C^{\Shalf,ss}$ for $\delta=1$ which parametrizes equivalence classes of semi-stable objects in $\cF_C^{\Shalf}$.

\subsection{Functorial properness of the evaluation map}
 In this section we prove a `$\Gamma$-equivariant analogue for quadratic-framed-modules of Proposition 2.8 of \cite[]{bs}' by evaluating at many points.
 
In the following proposition we take $G=\GL_n$ and $H=O_n$. Notice that for any $\delta$ a semi-stable framed module $f: W \ra \Shalf$ cannot admit non-zero maps from line bundles of arbitrary degree. An explicit bound for degree is $\delta |\Shalf| (1 - \frac{1}{\rank})$. Similarly $f_1: \Sym^2W \ra S_1$ is semi-stable for $\delta=1$ if $f: W \ra \Shalf$ is so. So $\ker(f_1)^*$ is a semi-stable vector bundle of positive degree. Thus again it is possible to give the bound of $\frac{\deg(\ker(f_1)^*)}{\rank(\ker(f_1))}$ such that $\Sym^2W^*$ does not admit non-zero maps from line bundles of higher degree. We call such an integer $m$ in the proposition below.

\begin{prop} \label{piFgh} Let $p: Y \ra X$ be a Galois cover of smooth projective curves with Galois group $\Gamma$. Let $F_{H,G}: \{Schemes \}  \ra  \{ Sets \}$ be the functor that to a scheme $T$ associates the set of isomorphism classes of pairs $(f_T: W_T \ra p_T^*\Shalf,s_T)$: where $f_T:W_T \ra p_T^*\Shalf = \{f_t: W_t \ra \Shalf \}_{t \in T}$  is a $T$-family of semi-stable $\Gamma$-framed modules on $Y$  and $s_T$  is a $\Gamma$-equivariant section of $\Sym^2W_T^* \ra Y \times T^*$ that is everywhere non-degenerate. In other words, it is an orthogonal bundle on $Y \times T^*$.

Let  $\cR \subset X$ be a collection of $m$ point outside of the branch locus of $p: Y \ra X$ and the support of $\Shalf$. Let $E(G/H)_\cR  \ra Y_\cR \times T$ be defined by the following diagram

\begin{equation} \xymatrix{\Sym^2W^*_\cR  \ar[r] \ar[d] & \Sym^2W^*_T \ar[d] \\ Y_\cR \times T \ar[r] \ar[d] & Y \times T \ar[d] \\ \{\cR \} \times T \ar[r] & X \times T } \end{equation}

Let $ F_{H,G,\cR}: \{Schemes \}  \ra  \{ Sets \}$ be the functor that to a scheme $T$ associates the set of isomorphism classes of pairs $(f_T:W_T \ra p_T^*\Shalf,\sigma_\cR)$ where $f_T: W_T \ra p_T^*\Shalf =\{f_t: W_t \ra \Shalf \}_{t \in T}$  is a $T$-family of $\delta$-semi-stable framed modules on $Y$ and $\sigma_\cR$ is a $\Gamma$-equivariant section of   $\Sym^2W^*_\cR  \ra Y_\cR \times T$ that is everywhere non-degenerate. In other words, it is an orthogonal bundle on $Y_\cR \times T$.

Then the morphism $\alpha_\cR: F_{H,G} \ra F_{H,G,\cR}$  induced by evaluation of sections at $Y_\cR$ is a proper morphism of functors. \end{prop}

\begin{proof} Let $T$ be a reduced affine curve and $p \in T$ be an arbitrary point. We denote $T^* = T \setminus \{ p \}$ denote the complement of the point $p$. We should show that given a $T$-family of $\delta$-semi-stable framed modules  $f_T: W_T \ra p_T^*\Shalf$  on $Y \times T$, and a $\Gamma$-equivariant every-where non-degenerate section $s_{T^*}$ of $\Sym^2W^*_{T^*}  \ra Y \times T^*$ together with a $\Gamma$-equivariant everywhere non-degenerate section $\sigma_T$ of $\Sym^2W^*_\cR \ra Y_\cR \times T$ such that for all $p_1 \in T^*$ we have $$s_{T^*} ( \{ y \} \times \{ p_1 \} ) = \sigma_T (\{ y \} \times \{ p_1 \}) \forall y \in Y_\cR $$ then there exists a $\Gamma$-equivariant section $s_T$ of $\Sym^2W^*_T \ra Y \times T$ prolongating $s_{T^*}$.

Now $s_{T^*}$ gives rise to a section of $\Sym^2 W^*_{T^*} \ra Y \times T^*$ which if it doesn't extend to a regular section $s'_T$ on $Y \times T$ would have a pole of order $k \geq 1$ at the divisor $Y \times \{ p \} \hra Y \times T$. Let $u$ denote the uniformizing parameter at the local ring $\cO_{T, p}$. Then $u^k s_{T^*}$ extends to a regular section (again denoted) $s'_T$ of $\Sym^2W^*_T \ra Y \times T$ whose restriction to $s'_{T, p}$ would be non-zero. Since $\Sym^2W^*_p \hra \ker(f_1)$ is semi-stable of though of positive degree so $s'_{T,p}$ vanishes at finitely many points in the second case. 

\begin{lem} The section $s'_T$ is $\Gamma$-equivariant. \end{lem} \begin{proof} In this proof, we shall view $W_T$ as a principal $\GL_n$-bundle. Since the property of separatedness is local on the base,  the structural map $W_T \ra Y \times T$ is separated because in the Zariski topology, it is a local $\GL_n$-fibration. Therefore $W_T(\GL_n/O_n) \ra Y \times T$ is also separated being a $\GL_n/O_n$-fibration Zariski locally. Thus the structural morphism of the $|\Gamma|$-fold fibered product of $W_T(\GL_n/O_n)$ with itself over $Y \times T$ is also separated. Let $\{ m_g \}_{g \in \Gamma}$ denote the lift of the action of $\Gamma$ on $Y $ to $E$ and let $\ol{m_g}$ denote the induced lifts on $E_T(\GL_n/O_n)$. Now since $Y \times T$ is reduced and since the section $(\ol{m_g}^{-1} s'_T g)_{g \in \Gamma}$ and the diagonal section $(s'_T)_{g \in \Gamma}$ agree on $Y \times T^*$ which is open and dense, so they agree on $Y \times T$. In particular, $s'_{T,p}$ on $Y \times \{ p \}$ is $\Gamma$-equivariant. \end{proof}
  
If the order of the pole $k$ were strictly greater than zero, then by continuity for all $p_1 \in T$, we would have $$ s'_T(\{ y \} \times \{p_1\}) = u^k \sigma_T (\{y \} \times \{ p_1 \}) \forall y \in Y_\cR.$$ This equality would imply that for $y \in Y_\cR$, we have $ s'_T(\{ y \} \times \{p \}) = u^k \sigma_T (\{y \} \times \{ p \})=0$  because $u$ vanishes at $p$ and $\sigma_T$ is regular. This is a contradiction owing to the size of $\cR$ that we took. So it must be that $s_{T^*}$ extends as a regular section to $s'_T$. Now since $W_T(\GL_n/O_n) \hra \Sym^2W_T$ is a closed embedding, so the entire image of $s_T$ lies in $W(\GL_n/O_n)$. In particular, the image of $s'_{T,p}$ lies in $W(\GL_n/O_n)_p$. It is also $\Gamma$-equivariant. This is what was required to be shown. \end{proof}

\begin{rem} \label{extfuncpropofev} For the case of quadratic bundles, the same proof as above also shows the properness of the evaluation of the quadratic form at sufficiently many points. We omit the proof.
\end{rem}

\subsection{Construction of quotient space for $\delta=1$}
Huybrechts-Lehn \cite[HL]{hl} have considered the closed sub-scheme of $Z'$ of $$\Quot_Y(H \otimes  \cO_Y(-m), P) \times \Proj_Y(\Hom(H, H^0(X, \Shalf)))$$  such that for $(q,p) \in Z'$ the following factorization of arrows takes place
\begin{equation}
\xymatrix{
H \otimes \cO_Y(-m) \ar[r]^q \ar[d]^p & \cF_q \ar@{-->}[ld] \\
\Shalf
}
\end{equation}
To simplify notation we shall denote $Z'$ the  locus stable under $\Gamma$-action in the description above. Note that over $Z'$ we have a universal map $f_{univ}: \cF \ra \Shalf$.
 
Let $f: W \ra \Shalf \in Z'^{ss}$ and consider $\ker(f)$. Then by the boundedness of the family $Z'^{ss}$ it follows that there exists an integer $m_0$ (independent of $\ker(f)$) such that for any non-zero section $s$ of $\Sym^2\ker(f)^*$, we have
$ |zeros(s)| < m_0.$

In the following we fix a $\Gamma$-invariant subset $J \subset Y$ outside of ramification locus of $p: Y \ra X$ and the support of $S$ having cardinality $|J|$ greater than $ m_0$.
Let $\cE$ denote the universal sheaf $\ker(f_{univ})$ on $Y \times Z'^{ss}$. It is a semi-stable vector bundle of negative degree by Propositions \ref{kerstable}. Let $\cE_G$ denote the associated principal $\GL_n$-bundle. We set
\begin{equation} Q' = (\Sym^2\cE^*)_J = \prod_{j \in J} \Sym^2\cE^*_{y_j} \ra Z'^{ss} \end{equation}
the fiber product taken for all $y_j \in J$ over $Z'^{ss}$. Let $f': Q' \ra Z'^{ss}$ be the natural map. It is an {\it affine} morphism because each $\Sym^2\cE^*_{y_j} \ra \{ y_j \} \times Z'^{ss} \stackrel{\simeq}{\ra} Z'^{ss}$ is affine being pull-back by $\{y_j\} \times Z'^{ss} \hra Y \times Z'^{ss}$ of $\Sym^2\cE^* \ra Y \times Z'^{ss}$, which being a vector bundle is affine. 

Note that $Q'$ parametrizes a $\delta=1$ semi-stable family of framed modules $f: W \ra \Shalf$ together with "initial values of the quadratic form" on $\Sym^2\ker(f)^*$ which is a semi-stable vector bundle of positive degree by Proposition \ref{2ndsymresult}.

Let $q'': \{ Schemes \} \ra \{ Sets \}$ be the functor that to a scheme $T$ associates 
$\{(f_T: W_T \ra p_T^* \Shalf, q_T) \}$ where $q_T$ is an everywhere non-degenerate quadratic form and $f_T$ is a $\delta=1$-semi-stable family of $\Gamma$-framed modules. We shall view $q_T$ equivalently as sections of $\Sym^2W_T^*$ and further as sections of $\Sym^2\ker(f_T)^*$. By the theory of Hilbert schemes (cf Lemma 3.8.1 \cite[]{ar}), and simple base-change we can show that $q''$ is representable by a $Z^{'ss}$ scheme $Q''$.

Let us define the evaluation map of $Z'^{ss}$-schemes by
$$ev_J: Q'' \ra Q', (f_T, q_T) \ms \{f_T, q(j) | j \in J \}.$$

\begin{lem} \label{lemma2} The evaluation map $\phi_J: Q'' \ra Q'$ is affine for $|J|$ large.
\end{lem}
\begin{proof} Suppose that $\phi_J(f_T,q_T) = \phi_J(f'_T,q_T')$. Then we may assume that $f_T: W_T \ra \Shalf = f'_T: W'_T \ra \Shalf$ and that $q$ and $q'$ are two different sections of $\Sym^2\ker(f_T)^*$ which agree on $J$. Since $\Sym^2\ker(f_T)^*$ is a family of semi-stable vector bundles, so it follows that $q=q'$. The properness follows from Prop \ref{piFgh}. Thus the evaluation map $\phi_J$ being proper and injective is affine.
\end{proof}

Let $\cG$ denote $\GL^{\Gamma}(H)$. By our definition of semi-stability of $\Gamma$-orthogonal bundles with frames, it is immediate that the $\cG$ action on $Z'^{ss}$ lifts to $Q''$. Further we have the commutative diagram
\begin{equation}
\xymatrix{
Q'' \ar[r]^{\phi_J} \ar[rd]^\mu & Q' \ar[d]^f \\
& Z'^{ss}
}
\end{equation}
By Lemma \ref{lemma2}, we have $\phi_J$ is affine. Further $f$ is affine because its fibers are $(\Sym^2U)^{J}$ where $U$ is any vector space of dimension the rank of the framed modules. Hence $\mu$ is a $\cG$-equivariant affine morphism.

\begin{prop} Two semi-stable $\Gamma$-quadratic framed modules $(f,q)$ and $(f',q')$ are in the same $\cG$-orbit of $Q''$ if and only if they are isomorphic.
\end{prop}
\begin{proof} The main ingredient we use is that the moduli of $\Gamma$-framed modules is the GIT quotient of $Z'$ under $\cG$. Thus the isotropy group at $f: W \ra \Shalf$ identifies with its automorphism group. 

Now suppose that $(f,q)$ and $(f',q')$ are in the same $\cG$-orbit. Then firstly $f: W \ra \Shalf$ and $f': W' \ra \Shalf$ are isomorphic as $\Gamma$-framed modules. Identifying them, $q$ and $q'$  as sections of $H^0(Y,\Sym^2W^*)$ lie in the same orbit of the isotropy group of $f: W \ra \Shalf$ which identifies with its automorphism group as a $\Gamma$-framed module. Hence $q$ and $q'$ furnish isomorphic $\Gamma$-quadratic bundles with frames. 

Conversely if $(f,q)$ and $(f',q')$ are isomorphic then firstly an element of $\cG$ takes the underlying $\Gamma$-framed module to the other. Identifying them, automorphism group identifies with the isotropy group on the one hand, and  it contains an element taking $q$ to $q'$ on the other.
\end{proof}

\begin{thm} \label{cms} There exists a coarse moduli space for semi-stable $\Gamma$-quadratic bundles with frame structure which is a quasi-projective variety. Further the map forgetting the quadratic structure to the coarse moduli space of semi-stable $\Gamma$ framed modules is affine.
\end{thm}
\begin{proof} 
%The space of surjective morphisms in $\Proj_Y(\Hom(H, H^0(X, \Shalf)))$ form an open subset $U$. 
Consider the $\cG$ equivariant affine morphism $\mu: Q'' \ra Z'^{ss}$. Since the quotient of $Z^{'ss}/\cG$ exists, so by Ramanathan's lemma 4.1 in \cite[]{ar} the quotient of $Q''/\cG$ also exists.  
By the universal property of categorical quotients, the canonical morphism  $\ol{\mu}: Q''/\cG \ra Z^{'ss}/\cG$ is also affine.
\end{proof}

%Since $Q'' \ra Q''/\cG$ is a good quotient, so the fiber of each $\Gamma$-framed module $f: W \ra \Shalf$ contains a unique closed $\cG$-orbit. Let $f: W \ra \Shalf$ deform to $(gr(W),gr(\alpha))$. Note that $(gr(W))$ consists of 
 
\subsection{Points of moduli when $S \neq \{0\}$} \label{pointsmoduli}
We separate a few lemmas, before describing the points of the quotient space $Q''/\cG$.
We remind the reader that the underlying sheaf $W$ in this section has degree zero.

\begin{lem} \label{degperp} Let $(W,q')$ be a quadratic vector bundle of degree zero on a curve $Y$. For any sub-bundle $F \subset W$, we have $\deg(F^\perp)=\deg(F)$.
\end{lem}
\begin{proof} This follows from the short exact sequence, 
$$0 \ra F^\perp \hra W \stackrel{q'}{\lra} W^* \ra F^* \ra 0.$$
Here we use that $\deg(F^*)=-\deg(F)$ because $F$ is a sub-bundle.
\end{proof}

\begin{Cor} For a sub-bundle $F \subset W$, we have $F=(F^\perp)^\perp$ if $W$ is an orthogonal bundle.
\end{Cor}
\begin{proof} The inclusion $\subset$ is clear and the equality follows by considering degrees and applying Lemma \ref{degperp}.
\end{proof}

By {\it destabilizing} subsheaf we mean a sheaf of slope equal or more than the ambient one.

\begin{prop} \label{saturated} Let $S \neq \{0\}$. Let $(f: W \ra \Shalf, q')$ be a strictly $\delta=1$-semi-stable quadratic bundle with frames. Let $\tilde{W_1} \subset W$ be any destabilizing subsheaf. Then it embeds in a saturated destabilizing sub-sheaf $W_1$, i.e a sub-bundle. Further, $W_1^\perp \subset W_1$ and the framing on $W_1^\perp$ is non-zero. 
\end{prop}
\begin{proof} 
Suppose that $\tilde{W_1}$ is a destabilizing sub-sheaf of $f: W \ra \Shalf$ with induced framing zero. Then if $K=\ker(f)$ then $\tilde{W_1} \hra K$. Without loss of generality we may assume that $\tilde{W_1}$ is a sub-bundle of $K$ else we may take its saturation. Now observe that $\tilde{W_1}$ is not saturated as a sub-sheaf of $W$. Applying $\Hom_{\cO_Y}(?,\Shalf)$ to the exact sequence $0 \ra W_1 \ra K \ra Q \ra 0$, we obtain 
$\Ext^1(\Shalf,K) \ra \Ext^1(\Shalf,W_1)$ and let $0 \ra \tilde{W_1} \ra W_1 \ra \Shalf$ be the image of $0 \ra K \ra W \ra \Shalf$. Observe that $W_1 \hra W$ is a sub-bundle with 
the same slope as $\tilde{W_1}$ as {\it framed modules} because $\delta=1$. {\it Thus we may restrict to sub-sheaves of $W$ with non-zero framing to check semi-stability and we will do so in this sub-section.}

Now suppose that the framing restricted to $W_1$ is not zero. Then without loss of generality we may suppose that $W_1$ is a sub-bundle because saturation only increases framed-slope.  By the exact sequence
\begin{equation} \label{ramtrick}
0 \ra W_1 \cap W_1^\perp \ra W_1 \oplus W_1^\perp \ra W_1 + W_1^\perp \ra 0,
\end{equation}
applying Lemma \ref{degperp} to $W_1$ and $W_1 + W_1^\perp$, it follows that $\deg(W_1 + W_1^\perp)=\deg(W_1)$ because $W_1 +W_1^\perp=(W_1 \cap W_1^\perp)^\perp$. Now, since $\deg(W)=0$ and its framed-slope is strictly negative because we have assumed that $S \neq \{0 \}$, if $\rank(W_1 + W_1^\perp) > \rank(W_1)$ then the bundle $W_1 + W_1^\perp$ would contradict semi-stability of $W$. So, $W_1^\perp \subset W_1$ and the framing on $W_1$ is non-zero. Further by equality 
\begin{equation}
\frac{\deg(W_1)- |S|}{\rank(W_1)} = \frac{0 - |S|}{\rank(W)},
\end{equation}
it follows that $\deg(W_1) \geq 0$. Since $\deg(W_1^\perp)=\deg(W_1)$, it follows that the framing on $W_1^\perp$ must be non-zero else, $W_1^\perp$ would contradict semi-stability.

\end{proof}
Note that if $S=\{0\}$ then it may happen that $W^\perp \subset W$ is destabilizing as it happens for  semi-stable quadratic bundles or even $\rank(W_1)+\rank(W_1^\perp) > \rank(W_1)$ for semi-stable orthogonal bundles.

\begin{rem} \label{cc} By the proof of Prop \ref{saturated} it follows that  checking semi-stability reduces to checking over familiar test objects consisting of  sub-bundles $W_1$ such that $W_1^\perp \subset W_1$ (cf \cite{sr}). With trivial framed structure, and when the quadratic form is non-degenerate everywhere, we recover the condition for orthogonal bundles.
\end{rem}

\subsubsection{S-equivalence classes} \label{sequiclass}
We now want to describe the equivalence relation on semi-stable $\Gamma$-quadratic bundle with frames. Let $(f: W \ra \Shalf,q')$ be such a bundle. Then by Prop \ref{jh}, ignoring the form $q'$, we have a filtration 
$0 \subset W_1 \subset W_2 \subset \cdots W$
of $W$ whose successive quotients are $\Gamma$-stable framed modules. If $W_1$ is not saturated, then we make it so as in the second part of the proof of Prop \ref{saturated}. Note its saturation continues to be $\Gamma$-stable of the same slope. Now $W/W_1$ is a semi-stable $\Gamma$-framed module with induced framing zero and it may be filtered by a Jordan-Holder sequence. Let us denote such a filtration by $0 \subset W_1 \subset W_2 \subset \cdots \subset W$. Now by Proposition \ref{saturated}, we have that $$0=W^\perp_{n+1} \subset  W^\perp = W^\perp_n \subset W_{n-1}^\perp \subset \cdots \subset W_2^\perp \subset W_1^\perp \subset W_1 \subset W_2 \subset \cdots W$$
because all $W_i$ have non-zero framing. Note that $W_1/W_1^\perp$ is an orthogonal bundle and we have {\it perfect pairings} between $W_i/W_{i-1}$ and $W_{i-1}^\perp/W_i^\perp$ if $i \leq n$. This  makes $W_i/W_i^\perp$, for $i \leq n$, into an orthogonal bundle. Now we may deform, as in Prop \ref{jh}, $W$ to $W_1 \oplus \oplus_{i \geq 2} W_i/W_{i-1}$. This splits $W$ further as \begin{equation} \label{S-equivalence}
W^\perp \oplus \oplus_{i \geq 1} W^\perp_i/W_{i+1}^\perp \oplus W_1/W_1^\perp \oplus_{i \geq 1} W_{i+1}/W_i.
\end{equation} As in Prop 1.13 \cite{hl}, any other Jordan-Holder filtrations $\{0\} \subset \tilde{W_1} \subset \tilde{W_2} \subset \cdots \subset W$ of the underlying $\Gamma$-framed module $W$ by destabilizing framed sub-modules $\tilde{W_i}$ only permute the successive quotients $W_i/W_{i-1}$. This has the effect of permuting the factors in sequence (\ref{S-equivalence}) together with the perfect pairings. 

\begin{rem} Note that in Prop \ref{jh} or Prop 1.13 \cite{hl}, only one successive quotient is endowed with non-zero framing. In the description above we have fixed the one with non-zero framing as $W_1$. 
\end{rem}
Now we make this into a definition using Proposition \ref{saturated} according to which we only need to know $q$ and the filtration $W_1 \subset W_2 \cdots$. The bundles orthogonal to them can be recovered.

\begin{defi} \label{S-def} We say that two semi-stable $\Gamma$-framed quadratic bundles $W$ and $\tilde{W}$ are $S$-equivalent if there exists filtrations $\{0\} \subset W_1 \subset W_2 \subset \cdots \subset W$ and $\{0\} \subset \tilde{W_1} \subset \tilde{W_2} \subset \cdots \subset \tilde{W}$ by $\Gamma$-stable framed modules which are {\it saturated} sub-bundles such that the associated graded modules are equivalent (upto permutation). 
\end{defi}

\begin{defi} \label{polystable} We call a framed module $(f: W \ra \Shalf,q')$ $\Gamma$-polystable if $W \simeq \oplus_{i \geq 1} W^\perp_i/W_{i+1}^\perp \oplus W_1/W_1^\perp \oplus_{i \geq 2} W_i/W_{i-1}$ and such that for any $m$, $$\oplus_{i \geq 1} W^\perp_i/W_{i+1}^\perp \oplus W_1/W_1^\perp \oplus_{m \geq i \geq 2}W_{i}/W_{i-1} $$ is a $\delta=1$-semi-stable {\it saturated} $\Gamma$-framed module.
\end{defi}

Note that as per our definition, stable objects are those whose underlying $\Gamma$-framed module is {\it stable}. In the case, when $\Gamma$ is trivial, the framing is trivial, and the quadratic form is non-degenerate, we get an orthogonal bundle. As per our definition, it is stable if the underlying vector bundle is stable. A usual stable orthogonal bundle is $S$-equivalent to a polystable framed quadratic bundle. But this would be sufficient to show that polystable orbits are closed.

\begin{defi} \label{morgen} We define morphisms between $\Gamma$-framed quadratic modules $(f_i: W_i \ra \Shalf,q'_i)$ (for $i=\{1,2\}$) as a morphism $\theta: W_1 \ra W_2$ between their underlying bundles such that there exists constants $c_S$ and $c_q$ such that $f_2 \circ \theta = c_s f_1$ and $c_q q_1'= \theta^* q_2' \theta$.
\end{defi}

We may prove an analogue of Lemma \ref{1.6}, stating that non-zero morphisms of stable $\Gamma$-quadratic bundles with frames must be isomorphism given by multiplication by scalars. This follows for free since multiplication by scalars will respect the quadratic forms.

We now wish to generalize Lemma \ref{scl} to show polystable orbits are closed.

\begin{lem} \label{scl2} Let $(f_1: W_1 \ra S, q_1)$ and $(f_2: W_2 \ra S,q_2)$ be flat families of $\Gamma$-semi-stable framed quadratic modules  parametrized by a scheme $T$ of finite type over $\CC$ with the same reduced Hilbert polynomial $p$. Then the function
$$t \ms \dim_{k(t)} \Hom_\Gamma ((f_1: W_1 \ra S, q_1),(f_2: W_2 \ra S,q_2)) $$
is upper semi-continuous in $t \in T$.
\end{lem}

\begin{proof} Let $F_T$ be a coherent $\Gamma$-sheaf on $Y_T = Y \times T$. Now we have functors \begin{equation} \xymatrix{ \{ \Gamma \text{-sheaves on} Y \} \ar[r]^{\Gamma_Y} & \{\Gamma \text{-modules} \} \ar[r]^\Gamma & \{ \text{Abelian groups} \} \\ F \ar[r] & H^0(Y,F) \ar[r] & H^0(Y, F)^\Gamma } 
\end{equation}
 \begin{equation} \xymatrix{ \{ \Gamma \text{-sheaves on} Y \} \ar[r]^{p_*^\Gamma} & \{\text{sheaves on} X \} \ar[r]^{\Gamma_X} & \{ \text{Abelian groups} \} \\ F \ar[r] & p_*^\Gamma(F) \ar[r] & H^0(X, p_*^\Gamma( F)) } 
\end{equation}
Recall that $p_*^\Gamma(F)(U) = F(p^{-1} U)^\Gamma$. Let us firstly show that $H^i(Y_T, \Gamma, F_T) = H^i(X, p_*^\Gamma(F))$.  We have a spectral sequence $R^p p_* \circ R^q \Gamma (F) => R^{p+q} (p_*^\Gamma)(F)$. But in characteristic zero, taking $\Gamma$-invariants is an exact functor, so $R^q \Gamma =0$ for $q>0$. And  since $p$ is finite, so  $p_T$ is finite and thus it is affine. Thus $R^p p_* =0$ for $p>0$. Thus $R^n(p_*^\Gamma)=0$ for $n>0$.  

Thus $H^i(X_T, p_*^\Gamma(F_T)) = R^i(\Gamma_X \circ p_*^\Gamma) (F_T) = R^i(\Gamma \circ \Gamma_Y) (F_T)$ which by definition equals $H^i(Y_T, \Gamma, F_T).$ Thus one has the Grothendieck complex $P^\bullet$ of finite free $A$-modules bounded above and which calculates the cohomology of $H^i(Y_T, \Gamma, F_T)$. Therefore using $P^\bullet$ and following the arguments in \cite[chapter 3, Section 12]{hart}, we see that the function $t \ms \dim_{K(t)} H^i(Y_t, \Gamma, F_t)$ is upper semi-continuous.

Step 2: Since the question is local on the base $T$, so we may suppose that $T = \Spec(A)$. Now suppose that $\theta: W_1 \ra W_2$ is a $\Gamma$-homomorphism of flat families. Then there are an induced morphism 
\begin{eqnarray}
p_*^\Gamma(\Hom(W_1,W_2)) \ra p_*^\Gamma (\Hom(W_1,S)) \\
p_*^\Gamma(\Hom(W_1,W_2)) \ra p_*^\Gamma (\Hom(W_1,W_1^*))
\end{eqnarray}
given by composing with $f_2:W_2 \ra S$ and $\theta \ms \theta^* \circ q_2 \circ \theta$ respectively.

Let $P^\bullet_{W_2}$, $P^\bullet_{W_1^*}$ and $P^\bullet_S$ be the Grothendieck complexes of $p_*^\Gamma(\Hom(W_1,W_2))$, $p_*^\Gamma (\Hom(W_1,W_1^*))$ and $p_*^\Gamma (\Hom(W_1,S))$ respectively. Recall that these complexes have the property that for any $A$-module $M$, the following diagram commutes \begin{equation} \xymatrix{ h^i(P^\bullet_{W_2} \otimes M) \ar[r] \ar[d] & h^i(P^\bullet_S \otimes M) \ar[d] \\ H^i(Y_T, \Gamma, \cH om(W_1, W_2 \otimes M)) \ar[r] & H^i(Y_T, \Gamma, \cH om(W_1, S \otimes M)) } 
\end{equation} 
Further they behave well under arbitrary base change. More precisely, if $M=B$ is an $A$-algebra and $g: T'=\Spec(B) \ra T$, then we have canonical isomorphism
\begin{equation}
H^i(Y_T, \Gamma, \cH om(W_1, U \otimes M)) \stackrel{\simeq}{\ra} H^i(Y_{T'}, \Gamma, \cH om_{T'}(W_1 \otimes M, U \otimes M))
\end{equation}
where $U=\{W_2,S\}$.

Step 3: We now finish the proof. Let $M$ be an $A$-module that we will set in a moment. We use notation as above. 

Let $\tilde{f_1} \in P^0_{S_T \otimes M}$ be the cycle that represents the framing \begin{equation}
f_1 \in \Hom_{\Gamma-A}(W_1, S_T \otimes M) = h^0(P^\bullet_{S_T}). 
\end{equation} Then $\tilde{f_1}$ gives a chain homomorphism  \begin{equation}
\tilde{f_1} : A^\bullet \ra P^\bullet_D
\end{equation}
 where $A^\bullet$ is the complex concentrated in zero with $A^0=A$ and $A^i=0$ for $i \neq 0$. Similarly let $\tilde{q_1} \in P^0_{W_1^* \otimes M}$ be the cycle that represents the 
 \begin{equation}
q_1 \otimes \Id_M \in \Hom_{\Gamma-A}(W_1 \otimes M, W_1^* \otimes M) = h^0(P^\bullet_{W_1^* \otimes M}). 
\end{equation}
 Consider the homomorphism 
 \begin{equation}
 \psi = ((f_2,q_2), (-\tilde{f_1},-\tilde{q_1})) : P^\bullet_{W_2 \otimes M} \oplus A^\bullet \ra P^\bullet_{S_T \otimes M} \oplus P^\bullet_{W_1^* \otimes M}
 \end{equation}
 where $q_2: P^\bullet_{W_2 \otimes M} \ra P^\bullet_{W_1^* \otimes M}$ is obtained by transposing with $q_2$ i.e $\theta \ms \theta^* \circ q_2 \circ \theta$. Let $C(\psi)$ denote the associated  mapping cone. From the  short exact sequence  of complexes
 \begin{equation}
 0 \ra P^\bullet_{S_T \otimes M} \oplus P^\bullet_{W_1^* \otimes M} \ra C(\psi)^\bullet \ra (P^\bullet_{W_2 \otimes M} \oplus A^\bullet) [1] \ra 0
 \end{equation}
  we get $0=$
  \begin{eqnarray*}
 H^{-1}(Y_T, \Gamma, \Hom(W_1 \otimes M, S \otimes M)) \oplus H^{-1}(Y_T, \Gamma, \Hom(W_1 \otimes M, W_1^* \otimes M)  \\ \ra
h^{-1}(C(\psi)^\bullet \otimes M)  \ra \Hom_{\Gamma}(W_1 \otimes M,W_2 \otimes M) \oplus M  \\ \ra
 \Hom_{\Gamma}(W_1,S_T \otimes M) \oplus  \Hom_\Gamma(W_1 \otimes M,W_1^* \otimes M).
  \end{eqnarray*} In particular for any $t \in \Spec(A)$ and $M=k(t)$ there is a pull-back diagram \begin{equation} 
  \xymatrix{ h^{-1}(C(\psi)^\bullet \otimes k(t)) \ar[rrr] \ar[d] &&& k(t) \ar[d]_{({f_1}_t,{q_1}_t)} \\ \Hom_{\Gamma-k(t)} ({W_1}_t ,{W_2}_t) \ar[rrr]^{{f_2}_t} &&& \Hom({W_1}_t, S \otimes k(t)) \oplus \Hom({W_1}_t,{W_1^*}_t)  }
  \end{equation} 
  
% Recall the definition of $\epsilon$ from Definition \ref{framedmoduledef}. By considering the cases $\epsilon(\alpha_t) =0 $ and $1$ individually, 
 We obtain 
  \begin{equation*} 
  \dim(\Hom((f_1:W_1 \ra S,q_1)_t,(f_2:W_2 \ra S,q_2)_t) = \dim h^{-1}(C(\psi)^\bullet \otimes k(t)).
  \end{equation*} 
% By assumption ${f_1}_t$ is zero either for all $t \in T $ or for none. 
This shows that the function stated in the lemma is upper semi-continuous.  
\end{proof}

\begin{prop} The points of $Q''/\cG$ are given by $\Gamma$-polystable framed modules with quadratic structure. 
\end{prop}
\begin{proof} Since the quotient $q: Q'' \ra Q''/\cG$ is a good quotient, so for each point $W \in Q''/\cG$, the fiber $q^{-1}(W)$ contains a unique closed $\cG$-orbit. By standard arguments we may deform $(f: W \ra \Shalf, q')$ to its associated $\Gamma$-polystable object. To prove that orbits of polystable objects are closed, we may follow the proof of Prop \ref{3.3}, working with Definition \ref{morgen}. The only non-trivial ingredient in this method is proved in Lemma \ref{scl2}. 
\end{proof}

\section{Semi-stable reduction for $\delta=1$ quadratic bundles with frames} \label{sssred}
A the risk of repetition, we have tried to make this main section of the paper self containted. Let us first introduce the notation that will be used in this section.

\begin{enumerate}
\item  $A$ a discrete valuation ring 
\item $\pi$ uniformizing parameter of $A$ 
\item $p$ maximal ideal of $A$
\item $T=\Spec(A)$ 
\item $T^*=\Spec(K)$
\item the depth of $S$ is $d$
\item $\Sigma$ is the surface $X \times T$ 
\item $\Sigma^*$  open subset $X \times T^*$ of $\Sigma$.
\item $Y$ is the vertical divisor $X \times p \hra S$ 
\item we will choose a collection of $\frac{|S|}{\rank(W)}$  points in $X$ and denote it as $x_0 \in X$
\item $H$ is the horizontal divisor $x_0 \times T \hra S$ 
\item $H^*$ is $H \cap \Sigma^*$ 
\item $P$ is the support in $X$ of the sky-scraper sheaf $S$
\item $m_P$ is the ideal of the stalk $\cO_P$ such that $\cO_P/m_P$ becomes a product of fields
\item $P_d$ is the nilpotent closed subscheme scheme in $X$ defined by the divisor $d P$
\item $D$ is the divisor $P \times T$
\item $D^*$ is $D \cap \Sigma^*$
\item $D_d$ is the nilpotent closed subscheme of $\Sigma$ defined by the divisor $d D$
\item we denote projection maps $A \times B \ra A$ as $q_B$ or $p_A$
\end{enumerate}

Recall that a quadratic bundle $(W,q)$ is a pair consisting of a vector bundle on $X$ of degree zero together with a quadratic form $q: \Sym^2 W \ra \cO_X$ such that the $\ker(q) \hra W$ is a vector bundle (possibly of rank zero) of degree zero.

We say that two $\cT$-families $(W_\cT^i,q_i, f_i: W_\cT^i \ra S)$, $i \in \{1,2\}$ of semi-stable quadratic bundles with frames   on $X \times \cT$ are {\it equivalent} if there exists a line bundle $M$ on $\cT$ of order two and an isomorphism

$$\theta: W_\cT^1 \ra W_\cT^2 \otimes p_\cT^* M,$$
of vector bundles on $X \times \cT$ such that the following diagram  commutes
\begin{equation}
\xymatrix{
W_\cT^1 \ar[rr]^{q_1} \ar[d]^{\theta} && W^{1*}_\cT \\
W_\cT^2 \otimes p_\cT^* M \ar[rr]^{q_2 \otimes 1_{p_\cT^* M}} && (W_\cT^2 \otimes p_\cT^* M)^*  \ar[u]^{\theta^*}
}
\end{equation} 
and there exists a non-zero scalar $\lambda \in \CC$ such that the following diagram also commutes
\begin{equation}
\xymatrix{
W_\cT^1 \ar[r]^\theta \ar[d]^{f_1} & W_\cT^2 \otimes p_\cT^* M \ar[d]^{f_2} \\
S \ar[r]^\lambda & S \otimes p_\cT^* M
}
\end{equation}

In practice, we will only need $\cT$ to be an affine spectrum of a DVR. In which case any line bundle $M$ must be trivial since the class group of $\cT$ is trivial being a UFD. The above condition reduces simply to demanding the commutation of the following diagram
\begin{equation}
\xymatrix{
W_\cT^1 \ar[rr]^{q_1} \ar[d]^{\theta} && W^{1*}_\cT \\
W_\cT^2   \ar[rr]^{q_2 } && (W_\cT^2 )^*  \ar[u]^{\theta^*}
}
\end{equation}
and that of 
\begin{equation}
\xymatrix{
W_\cT^1 \ar[r]^\theta \ar[d]^{f_1} & W_\cT^2 \ar[d]^{f_2} \\
S \ar[r]^\lambda & S 
}
\end{equation}

In particular, if $\cT= \Spec(A)$ for $A$ a DVR, and $ \lambda^2 \in A^\times$, the framed quadratic bundles $(V,q)$ and $(V,\lambda^2q)$ on $X \times \cT$ are equivalent.
 
\begin{thm} \label{ssred} The functor $F: \{Schemes \}  \ra \{Sets \}$ that associates to a scheme $\cT$ the set of equivalence classes of $\cT$-families of semi-stable framed quadratic bundles $(W_\cT,q, f^W: W_\cT \ra p_X^* S) \ra X \times \cT$ for $\delta=1$ is proper.
\end{thm}
\begin{proof}
We assume we are given a family $(W_{\Sigma^*},q^W_{\Sigma^*},f^W_{\Sigma^*}:W_{\Sigma^*} \ra p_X^* S)$ of semi-stable framed quadratic bundles.

For an arbitrary $\delta$-semi-stable $T$-family of framed modules (or level structures) $f_\Sigma: W_\Sigma \ra q^*_T S $ is known to pick up torsion i.e $W_Y$ may have torsion even when $W_{\Sigma^*}$ is torsion free.
In the following proposition, for $\delta=1$, using that $\ker(f_\Sigma)$ is a family of semi-stable vector bundles and properness of latter objects, we show that $W_Y$ will be
torsion-free. 

\subsection{Extension of framed module} \label{extn-ul-fm}
\begin{prop}  In the equivalence class of the family $(W_{\Sigma^*},f^W_{\Sigma^*})$ there is another family which extends to the divisor $Y \hra \Sigma$  as a semi-stable framed module for $\delta=1$ such that $W_Y$ is a vector bundle on $Y$.
\end{prop}
\begin{proof} Since $\delta=1$, we have two consequences:
\begin{enumerate}
\item the $\ker(f^W)=V_{\Sigma^*}$ is a family of semi-stable vector bundles (of negative degree) by Prop \ref{kerstable}.
\item for any $t \in T^*$, the framing maps $f^W_t$ are surjective by Prop \ref{delta=1surj}, and hence {\it non-zero}. Hence for any $t \in T^*$ we have a short exact sequence  $0 \ra V_t \ra W_t \ra S \ra 0$, or equivalenltly an element of $\Ext^1_X(S,V_t)$.
\end{enumerate}
We extend $V_{\Sigma^*}$ to a family $V_{\Sigma}$ on the whole of $\Sigma$ as a poly-stable vector bundle by properness of moduli space of polystable vector bundles of negative degree. We may trivialize $V_{\Sigma}$ in a neighbourhood of the divisor $D:=P \times T \hra \Sigma$ to fix an isomorphism of the {\it stalks} at $D$
\begin{equation} 
\psi_D: (V_{\Sigma})_D  \ra q_T^* \cO_P^n,
\end{equation} 
where $\cO_P$ denotes the stalk of $\cO_X$ at $P \hra X$. On the other hand, dualizing the sequence
\begin{equation}
0 \ra V_{\Sigma^*} \ra W_{\Sigma^*} \ra q_{T^*}^* S \ra 0,
\end{equation}
we get
\begin{equation} \label{dualeqn}
0 \ra W_{\Sigma^*}^* \ra V_{\Sigma^*}^* \ra \Ext_{\Sigma^*}^1(q_{T^*}^* S,\cO_{\Sigma^*}) \ra 0. 
\end{equation}
Now let us consider  the projection map $$p_{P_d}=q_{T^*}: D_d^*= P_d \times T^* \ra P_d.$$
One has the natural morphism 
\begin{equation} \label{locfree1} q_{T^*}^* \cE xt^1_X(S,\cO_X) \ra \cE xt_{\Sigma^*}^1(q_{T^*}^* S,\cO_{\Sigma^*}) 
\end{equation} which is checked to be an isomorphism. This can be seen as follows. We express the sky-scraper sheaf $S= \oplus_{1 \leq e \leq d} S_e$ as a direct sum of  sheaves $S_e$ which are {\it free modules over $\cO_{P_e}$} for $e \leq d$. The natural morphism 
\begin{equation} q_{T^*}^* \cE xt^1_X(S_e,\cO_X) \ra \cE xt_{\Sigma^*}^1(q_{T^*}^* S_e,\cO_{\Sigma^*}) 
\end{equation} is  an isomorphism because these sheaves are supported on $D_e^* \hra \Sigma^*$ where they are free. Note here that the depth of $\Ext^1_X(S,\cO_X)$ is the same as that of $S$, which is $d$. 

The trivialization $\psi_D$ restricted to $D^*$ gives us a morphism
\begin{equation}
\cH om_{D_d^*}(V_{\Sigma^*}^*|_{D_d^*}, q_{T^*}^* \cE xt^1_X(S,\cO_X)) \stackrel{\psi_D}{\lra} \cH om_{D_d^*}(q_{T^*}^* \cO_X^n|_{P_d^*},q_{T^*}^* \cE xt^1_X(S,\cO_X))
\end{equation}
Now again we have a natural morphism 
\begin{equation} \label{locfree2}
q_{T^*}^* \cH om_{P_d^*}(\cO_{P_d}^n, \cE xt^1_X(S,\cO_X)) \ra 
\cH om_{D_d^*}(q_{T^*}^* \cO_X^n|_{P_d^*},q_{T^*}^* \cE xt^1_X(S,\cO_X))
\end{equation} which is checked to be an isomorphism by decomposing $S=\oplus_e S_e$ as before where $S_e$ is supported on $P_e \hra X$. Then the summand corresponding to $S_e$ is supported on  $D_e^* \hra \Sigma^*$ where that summand is free. 

By the isomorphisms $\psi_D$, (\ref{locfree1}) and (\ref{locfree2}), by restriction to $D^*$  the morphism $V_{\Sigma^*}^* \ra \Ext_{\Sigma^*}^1(q_{T^*}^* S,\cO_{\Sigma^*})$ of (\ref{dualeqn}) furnishes a no-where vanishing  section 
 \begin{equation} H^0(D^*_d,q_{T^*}^* \cH om_{P_d}(\cO_{P_d}^n, \cE xt^1_X(S,\cO_X)))
 %= H^0(T^*,(p_{T^*})_* q_{T^*}^* \Hom_{P_d}(\cO_{P_d}^n, \Ext^1_X(S,\cO_X)))
 .
 \end{equation}
 By adjunction under the finite and affine projection morphism $p_{T^*}: D_d^* \ra T^*$, identifying $\cO_{D_d^*} = p_{T^*}^* \cO_{T^*}$ we obtain a global no-where zero morphism 
 \begin{equation}
 \cO_{T^*} \ra (p_{T^*})_* q_{T^*}^*  \cH om_{P_d}(\cO_{P_d}^n, \cE xt^1_X(S,\cO_X)).
  \end{equation}
 
We obtain a well-defined morphism \begin{equation} 
T^* \ra \PP(\Gamma(D_d,q_{T^*}^* \Hom_{P_d}(\cO_{P_d}^n, \Ext^1_X(S,\cO_X)))),
\end{equation}
where the projective space denotes the space of $1$-dimensional subspaces.
We extend this to a morphism $T \ra \PP$. Retracing the
 above steps backwards by replacing the roles of $T^*$ and $D^*_d$ by $T$ and $D_d$ respectively, 
 we obtain a $T$-family of non-zero extensions
 \begin{equation}
 0 \ra W^*_\Sigma \ra V^*_\Sigma \ra q_T^* \Ext^1_X(S,\cO_X) \ra 0
 \end{equation}
 which on dualizing defines
\begin{equation}
0 \ra V_\Sigma \ra W_\Sigma \ra q_T^*S \ra 0.
\end{equation}
Here we have used that $q_T^* \Ext^1_X(S,\cO_X)=\Ext^1_\Sigma(q_T^*S, \cO_\Sigma)$ and \begin{equation}
\Ext^1_\Sigma(\Ext^1_\Sigma(q_T^*S, \cO_\Sigma),\cO_\Sigma) = q_T^*S
\end{equation}
canonically because $S$ is a sky-scraper sheaf on $X$.

 Over $\Sigma^*$, the new projection maps $W_\Sigma \ra q_T^*S$ differ from $f^W_{\Sigma^*}$ by non-zero scalar $\lambda \in A^*$. Thus as a family of framed modules the new and old family are equivalent.

Note also since $V_Y$ is semi-stable as a vector bundle, so $f_Y: W_Y \ra S$ is semi-stable as a framed module for $\delta=1$. Dualizing the sequence 
\begin{equation}
0 \ra W_Y^* \ra V_Y^* \ra Ext^1(S,\cO_X) \ra 0
\end{equation} 
we get $0 \ra V_Y \ra (W_Y^*)^* \ra S \ra 0$. We take $W_Y$ to be $(W_Y^*)^*$. So it follows that $W_Y$ is torsion-free. 
\end{proof}

{\it We will replace the old family by the new in the rest of the proof and continue to denote the new projection maps also by $f^W_\Sigma$ to not introduce new notation.}

Note that by Theorem \ref{modspace}, only the $S$-equivalence class of $f^W_Y: W_Y \ra S$ is well-defined. Our remaining construction depends on this choice. 

%\begin{rem} We now begin a construction starting with $(W_\Sigma,f)$ to produce a semi-stable framed quadratic bundle $(U_\Sigma,q_\Sigma^U,f:U_\Sigma \ra S)$. If instead, we choose some other representative of its equivalence class, we will get a new framed module which is $S$-equivalent to $(U_\Sigma,q_\Sigma^U,f:U_\Sigma \ra S)$. alternatively we may also proceed as follows. The reader may assume $(W_\Sigma,f)$ to be poly-stable in the sense of semi-stable framed modules, and check that the final framed module is polystale in the sense of framed quadratic bundles in the limit i.e over $Y$.\end{rem}

\subsection{Extension of quadratic form} \label{extn-qf}
For the following proposition we choose a {\it single} arbitary base point $x_0 \in X$ and denote temporarily $H= x_0 \times T$.
The vector bundle $W_\Sigma$ restricted to the divisor $H  \hra S$ is the trivial module, say $W_{H}$, of rank $n$ over $A$. Further restriction to $H^* \hra H$ furnishes the map $$W_{H} \hra W_{H^*},$$
realizing $W_{H}$ as a lattice in $W_{H^*}$.

\begin{prop} There exists an appropriate basis $\{e_i^W \}$ of $W_{H}$, such that in the induced basis $\{e_i^W \otimes_A 1_K \}$ of $W_{H^*}$,  the quadratic form may be expressed over $H^*$ in terms of variables $X_i$ as the diagonal form
$$q^W_{\Sigma^*}(x_0)= \sum_i q^W_i(x_0,\pi) X_i^2: W_{H^*} \ra K$$
where $q^W_i(x_0,\pi) \in K \setminus \{ 0 \}$.
\end{prop}
\begin{proof}
Let $q^W_{\Sigma^*}(x_0)$ be the quadratic form on $W_{H^*}$ and let $B_{x_0}$ be the associated bilinear form {\it taking values in} $K$ on the {\it lattice} $W_{H}$ of $ W_{H^*}$. Choose a basis $\{ e_i \}$ of $W_{H}$. We define 
\begin{equation} \label{wmin} w_{min} = min_{B_{x_0}(e_i,e_j) \neq 0} val(B_{x_0}(e_i,e_j))
\end{equation}
This minimun exists because not all $B_{x_0}(e_i,e_j)$ are zero, since the form $q^W_{\Sigma^*}(x_0)$ is non-degenerate.
If $w_{min}$ is attained at a diagonal entry $B_{x_0}(e_i,e_i)$ (for some $i$) then we relabel our basis interchanging $e_1$ and $e_i$. Now suppose that $w_{min}$ is not attained at any diagonal entry, but say rather at $B_{x_0}(w_i,w_j)$. Then  we observe $$val(B_{x_0}(e_i+e_j,e_i+e_j))=val(B_{x_0}(e_i,e_j)),$$
because $val(B_{x_0}(e_i,e_j))$ is strictly less than $val(B_{x_0}(e_i,e_i))$ and $val(B_{x_0}(e_j,e_j))$. By a uni-modular transformation, it is possible to choose a new basis of $W_{H}$ with $e_1 :=e_i+e_j$. In conclusion, it is possible to choose a basis $\{ e_t \}$ of $W_{H}$ such that
\begin{enumerate}
\item $val(B_{x_0}(e_1,e_1)) \leq val(B_{x_0}(e_i,e_j))$ whenever $B_{x_0}(e_i,e_j) \neq 0$,
\item $B_{x_0}(e_1,e_1) \neq 0$.
\end{enumerate}
Now, put $e^W_1=e_1$ but replace $\{e_i\}_{i \geq 2}$ by $\{e'_i\}_{i \geq 2}$ where 
$$e'_i = e_i - \frac{B_{x_0}(e_i,e_1)}{B_{x_0}(e_1,e_1)} e_1.$$
Note that if $B_{x_0}(e_i,e_1) =0$ then $e'_i=e_i$, and otherwise $val(\frac{B_{x_0}(e_i,e_1)}{B_{x_0}(e_1,e_1)}) \geq 0$, which means $\frac{B_{x_0}(e_i,e_1)}{B_{x_0}(e_1,e_1)}) \in A$. It follows thus in both cases that replacing $\{e_i\}$ by $\{e^W_1 \cup e'_i\}$ is a unimodular operation, and thus $\{e^W_1 \cup e'_i\}$ continue to be a basis of the lattice $W_{H}$. Now the submodule spanned by $\{e'_i\}_{i \geq 2}$ is orthogonal to $e^W_1$ and we may apply the above process to it. Thus we will get a basis $\{e^W_i\}$ of the lattice $W_{H}$ in terms of which, the quadratic form $q'_K(x_0)$ becomes diagonal as asserted in the proposition.
\end{proof}

\begin{prop} The function $X \ra \mathbb{N}$ that to $x$ associates $w_{min}^x$ as in (\ref{wmin}) is generically constant and upper semi-continous.
\end{prop}
\begin{proof} Let $U_0$ be a small open affine neighbourhood around $x$ such that $W$ becomes trivial when restricted to it. We choose a basis $\{f_i \}$ of $W|_{U_0}$. In terms of $\{f_i \}$ the quadratic form is given as a symmetric matrix with entries $B^{U_0}(u, \pi)_{ij} $ which we write as rational functions in $\pi$ with regular coefficients as follows:
\begin{equation}
B^{U_0}(u, \pi)_{ij} = \frac{\sum_{k=0}^{n_{ij}} B^{{U_0},k}_{ij}(u) \pi^k}{\pi^{a_{ij}}}
\end{equation}
where $a_{ij} \geq 0$, $n_{ij}$ only depend on ${U_0}$, $B^{U_0,k}_{ij}$ are regular functions on ${U_0}$ and $B^{U_0,0}_{ij}$ is not the zero function. Now for $y \in {U_0}$ lying outisde the zero locus of $B^{{U_0},k}_{ij}$ for varying $k,i$ and $j$, we have  $$w_{min}^y = - max_{ij} a_{ij}$$ and for the remaining finitely many points in $z \in \cup_{ij} W(B^{{U_0},k}_{ij})$, $w_{min}^z$ is only greater. Using standard arguments of valuations and the property that $val(a+b) \geq min \{val(a), val(b)\}$ we see that the function $x \ms w_{min}^x$ is upper semi-continous.
\end{proof}
Hence by the above proposition, we may attach a number $w$, namely the generic value of $w_{min}^x$ as $x$ varies in $X$.

For any number $m$ greater than $\frac{|S|}{\rank(W)}$, we choose $m$ points in $X$ with $w_{min}^x=w$ not lying in the support of $S$. {\it Henceforth, we call this collection as our base point $x_0$}.

We make the simple observation that the restriction of $q^W_{\Sigma^*}(x_0)$ to the lattice $W_{H} \hra W_{H^*}$ is again the same diagonal form $$\sum_i q^W_i(x_0,\pi) X_i^2: W_{H} \ra K$$ with respect to basis $\{ e_i^W \}$.

By going to a cover of $T$, we may assume that the valuations of $q^W_i(x_0,\pi)$ with respect to the uniformizing parameter $\pi$ are even, say $2w_i$.  Without loss of generality, we may assume that the valuation of $q^W_1(x_0,\pi)$ is minimum. We set
\begin{equation}
q_K=\pi^{-2w_1}q^W_{\Sigma^*}.
\end{equation}
Therefore expressing $q_K(x_0,\pi)=\sum_i q_i(x_0,\pi) X_i^2$, all the coefficients have now become regular with that of $q_1(x_0,\pi)$ being a unit. We note that \begin{equation} \label{val}
val(q_i(x_0,\pi))=: 2v_i \geq 0,
\end{equation} and denote the regular extension of $q_K$ at the base point $x_0$ as
\begin{equation} \label{qatx0}
q_A(x_0).
\end{equation}

As a $T^*$-family of quadratic bundles $(W_{\Sigma^*},q^{W}_{\Sigma^*})$ and $(W_{\Sigma^*}, \pi^{-2 w_1}q^{W}_{\Sigma^*})$ are isomorphic, since we have only multiplied $q^{W_{\Sigma^*}}$ with a single scalar, namely $\pi^{-2 w_1}$. So we may work rather with the latter and {\it we shall do so in the remaining part of the proof}. 

\begin{prop} The quadratic form $q^{W}_{\Sigma^*}$ extends regularly on the whole of $S$ as a section $q^{W}_\Sigma$ of  $\Sym^2W^*_\Sigma$ over $ S$. Its restriction $q^{W}_Y: \cO_{Y} \ra \Sym^2W^*_Y$ over $Y$ to the vector bundle $W_Y$ is non-zero.
\end{prop}
\begin{proof} Either $q^{W}_{\Sigma^*}$ extends regularly or it has a pole of order $k \geq 1$ along the divisor $Y$. In the latter case, let us derive a contradiction. We consider $q^{k,{W}}_\Sigma:= \pi^k q^{W}_{\Sigma^*}$ which will extend as a non-zero section of $\Sym^2W^*_\Sigma$.
\begin{lem} \label{resnonzero} The restriction $q^{k,{W}}_Y$ of the section $q^{k,{W}}_\Sigma$ to the special fiber $Y$ is  also {\it non-zero}.
\end{lem}
\begin{proof} To see this, consider the $A$-module $\Gamma(S, \Sym^2 W^*_{S})$. It contains a rank one sub-module $M$ generated by the section $q^{k,{W}}_A$ by mapping $1 \in A$ to $q^{k,{W}}_\Sigma \in  M$. Let us denote the surjective map again by $q^{k,{W}}_\Sigma: A \ra M$. Tensoring by the residue field, the section $q^{k,{W}}_Y: k(A) \ra M/\pi M$ remains surjective. To check that $q^{k,{W}}_Y \neq 0$, it suffices to check that $M/\pi M \neq \{0\}$. Now $M \neq \{0\}$ because $q^{k,{W}}_\Sigma \neq 0$. Thus by Nakayama Lemma, it follows that $M \otimes_A A/\pi A \neq \{0\}$. It is therefore a one dimensional vector space generated by the residue field $k(A)$ of $A$. This  makes the map $q^{k,{W}}_Y: k(A) \ra M/\pi M$ an isomorphism. Consequently, $q^{k,{W}}_Y$ is {\it non-zero}.
\end{proof}
 Now this contradicts the semi-stability of $\Sym^2W^*_Y$ as follows. We have the following factorization
\begin{equation}
\xymatrix{
\cO_{Y} \ar[r]^{q^{k,{W}}_Y} \ar[d] & \Sym^2W^*_Y \\
\cO_{Y}(x_0) \ar@{.>}[ru] &
}
\end{equation}
because $q_{S}(x_0)$ extended as a non-zero quadratic form on $x_0$ to all of $T$ in (\ref{qatx0}) . Now composing 
$\cO_{Y}(x_0) \ra \Sym^2W^*_Y$ is a non-zero map contradicting the semi-stability of $\Sym^2V^*_Y$.

In conclusion $q^{W}_{\Sigma^*}$ itself extends as the section $q^{W}_\Sigma: \cO_{\Sigma} \ra \Sym^2W^*_{\Sigma}$. Moreover by repeating arguments of Lemma \ref{resnonzero}, we see that its restriction $q^W_Y: \cO_{Y} \ra \Sym^2W^*_Y$ is non-zero.
\end{proof}

Consider the sequence
\begin{equation}
W_Y \stackrel{q^W_Y}{\ra} W^*_Y
\end{equation}
and let us denote the kernel by $K_Y$ and the quotient by $Q_Y$. Note that since $q^W_Y \neq 0$, so $K_Y \neq W_Y$ and $W^*_Y \neq Q_Y$. Now the map 
\begin{equation}
W_Y/\ker(q^W_Y) \ra \ker(W^* \ra Q_Y)
\end{equation}
is generically an isomorphism being induced by $q^W_Y$ which is non-zero. Hence $\deg(W_Y/\ker(q^W_Y)) \leq 0$ which implies that $\deg(\ker(q^W_Y)) \geq 0$. Also $\deg(\ker(W^* \ra Q_Y)) \geq 0$ which implies that $\deg(Q_Y) \leq 0$. Hence by the equality 
\begin{equation}
\deg(\ker(q^W_Y)) - \deg(W_Y) + \deg(W_Y^*) - \deg(Q_Y)=0
\end{equation}
 it follows that $\deg(\ker(q^W_Y))=0$. Since $\ker(q^W_Y) \hra W_Y$ so it is torsion free and hence locally free since $Y$ is a curve.
 
Hence $(W_Y,q^W_Y,f^W_Y:W_Y \ra S)$ is the sought framed quadratic bundle which is $\delta=1$ semi-stable.

Finally note that in place of $f^W_Y: W_Y \ra S$, had we taken some other framed module as extension of $f^W_{\Sigma^*}:W_{\Sigma^*} \ra q_T^*S$, then they would be $S$-equivalent as framed modules by Theorem \ref{modspace}, and therefore $S$-equivalent as quadratic framed modules by Definition \ref{S-def}.
\end{proof}

By Lemma \ref{resnonzero}, it follows that $\rank(\ker(q^W_Y)) < \rank(W)$.
Note that when $\rank(\ker(q^W_Y)) >0$, then the induced framing on $\ker(q^W_Y)$ must be non-zero, else it will be destabilizing because the degree of $W_Y$ is also zero.

\begin{rem} \label{limitingobjects}
The induced map $\ol{q^W_Y}: W/\ker(q) \ra (W/\ker(q))^*$ being generically an isomorphism between vector bundles of degree zero, must be an isomorphism. Also the sequence
\begin{equation}
\ker(q^W_Y) \hra W_Y \ra S
\end{equation}
furnishes $W/\ker(q) \ra S'$ where $S'$ is a certain quotient of $S$. Reinterpreting this data over $X$ by taking invariant direct image of $(f^W_Y: W_Y \ra S,q^W_Y)$, we see that the compactification of vector bundles of fixed degeneracy locus and type is done by everywhere degenerate bundles whose degeneracy is smaller than that of $S$ in the following sense. Since $S'$ is a quotient of $S$, so the type of degeneracy it induces on $W/\ker(q)$ are smaller than those induced by $S$. 
\end{rem}

We finish this section with a simple observation which proves properness when the quadratic form has "singularities of order one".

\begin{thm} \label{depth1} The functor $F^{ss, \ZZ/2}_S$ is proper when the sky-scraper sheaf $S$ has depth one for small values of $\delta$. 
\end{thm}  
\begin{proof} Let $V(y)_{-1}$ denote the $-1$ eigen-space for the $\ZZ/2$-action on $V_y/m_y$. Then in multiplicity one, we have an isomorphism of sheaves 
$$S=  \oplus_{y \in \ram(p)} V(y)_{-1}.$$
So set-theoretically the functor $F^{ss, \ZZ/2}_S$ equals the functor $F^{ss,\ZZ/2}$ (or $F$ for short) which forgets the frame structure and takes semi-stability in the sense of Seshadri \cite[CSS]{css}. Moreover, since there are finitely many choices of $\rank(W_1)$ of sub-bundles $W_1$ and $\epsilon(W_1)$, for small $\delta$ the inequality (\ref{ssdeforfs}) holds if and only if $\deg(W_1) \leq 0$ holds. This means that  the underlying vector bundle $W$ is semi-stable since its degree is zero because it supports an every-where non-degenerate quadratic form. Recall that in \cite[CSS]{css} a $\ZZ/2$-vector bundle is defined to be semi-stable if the underlying vector bundle is semi-stable. So the semi-stable objects associated by  the functor $F^{ss, \ZZ/2}_S$ and  the functor $F^{ss,\ZZ/2}$ are in bijection. Now the claim follows from the properness of moduli of $\ZZ/2$-orthogonal bundles \cite{vbcss}. 
\end{proof}

\section{Sky-scraper sheaf $S$ as discrete invariant} \label{s-di}
\begin{prop} \label{disceteinvariant} Fix a section $s \in \Gamma(X,\det(V^*)^2)$. On the fiber $disc^{-1}(s)$, the sky-scraper sheaf $V^*/q(V)$ defines a discrete invariant.
\end{prop}
\begin{proof} We spell out what we should check. Let $T$ be an arbitrary  reduced affine complex curve. Let $(V_T,q_T)$ be a 
 $T$-family of semi-stable quadratic vector bundles such that
 \begin{enumerate}
 \item the associated discriminant section $\{\det(q_t)\}_{t \in T}$ is constant \item and $q_T$ restricted to $T^*$ satisfies 
\begin{equation}
0 \ra V_{T^*} \stackrel{q_{T^*}}{\lra} V^*_{T^*} \ra p_X^*S \ra 0
\end{equation} 
for some sky-scraper sheaf $S$.
\end{enumerate}
 Consider the quotient $Q$ defined by $0 \ra V_T \stackrel{q_T}{\lra} V^*_T \ra Q$ on $X \times T$. Then we should check that $Q_t \simeq S$ also. Now we may assume that $T$ is the spectrum of a discrete valuation ring containing complex numbers. 

The sky-scraper sheaf $Q$ admits a filtration defined as follows. Let $D$ denote the divisor of {\it multiplicity one} which is the support of the sky-scraper sheaf $S$.  Let $s_D: \cO_X \ra \cO_X(D)$ denote the natural section (which is well-defined upto scalar multiplication). Consider now the natural section $s_{D,T}: \cO_{X \times T} \ra \cO_{X \times T}(\{D\} \times T)$. Define quotient sheaves $Q_n$ and kernel sheaves $K_n$ by tensoring 
\begin{equation} \label{fss}
0 \ra \cO_{X \times T} \stackrel{ s_{D,T}^n}{\lra} \cO_X(n \{D\} \times T) \ra \cO_{nD} \otimes_\CC B \ra 0
\end{equation} by $Q$ to get
\begin{equation} \label{knqn} 
0 \ra K_n \ra Q \stackrel{\times s^n_{D,T}}{\lra} Q \ra Q_n \ra 0. 
\end{equation} 
Since $Q$ is $B$-flat and therefore free, so $K_n$ being sub-module of $Q$ are torsion-free, thus free and thus $B$-flat too.  Consider the map $K_n  \stackrel{\times s^{n-i}_{D,T}}{\lra} K_n$. Its kernel is naturally $K_{n-i}$. Its image is torsion-free being a sub-module of $K_n$ which is free, and thus its image is $B$-free. Thus its image is projective as a $B$-module and thus the inclusions of sky-scraper sheaves $K_{n-i} \hra K_n$ are split as $B$-modules. In particular $K_n \hra Q$ is split. Thus $Q_n$ is isomorphic to $K_n$ as $B$-modules. 

Let $F_i$ denote the successive quotient for the natural inclusion $K_{i-1} \hra K_i$. Note also that $F_i$ is $B$-flat. We thus record
\begin{equation} \label{niceobs}
F_n=K_n/K_{n-1} = \ker(Q_n \ra Q_{n-1}),
\end{equation} 
which will be used to recover $Q$ from $F_n$. We see this as follows. The inclusion $$\Img(s^{i-1}_D)/\Img(s^i_D) \hra \cO_X/\Img(s^i_D)$$ as a sub-module endows $\Img(s^{i-1}_D)/\Img(s^i_D)$ with an $\cO_X$-module structure. This can be used to endow $\oplus_{i \leq n} F_i \otimes_\CC \Img(s^{i-1}_D)/\Img(s^i_D)$ with the structure of a $\cO_{X \times T}$-module. Now observe that we have $Q_n \simeq \oplus_{i \leq n} F_i \otimes_\CC \Img(s^{i-1}_D)/\Img(s^i_D)$ as $\cO_{X \times T}$-modules.

Just as the $F_i$ are defined above starting with $Q$, similarly define $S_i$ starting with the sky-scraper sheaf $p_X^* S$ via the sequence (\ref{fss}).

\begin{lem} \label{atsppoint} Over $X \times T$, the sheaf $Q$ is isomorphic to $p_X^*S$.
\end{lem}
\begin{proof} By the semi-continuity theorem applied to the sheaves $F_i$ on $X \times T$, we get 
\begin{equation} \label{scont}
|F_{i,t}| \geq |F_{i,\eta}|=|S_{i,\eta}|=|S_{i,t}|
\end{equation}
for all $i$ where $t \in T$ is the special point. However we also have
\begin{equation*}
 \sum_i |F_{i,t}| = \deg(V^*_t) - \deg(V_t)= \deg(V^*_\eta) - \deg(V_\eta)= |S|= \sum_i |S_i| 
\end{equation*}
because $V_T$ and $V_T^*$ are flat families over $X \times T \ra T$.
Thus we must necessarily have equality in (\ref{scont}). Thus the sheaves $F_i$ and $S_i$ are free $\cO_D \otimes_\CC B$ modules of the same rank. Recall here that we had chosen $D$ to have multiplicity one at each divisor.

Choose arbitrary isomorphisms between them. Now onwards we change our point of view; we shall view $F_i= \ker( Q_i \ra Q_{i-1})$ and similarly $S_i$ as sub-quotients rather than sub-objects of $Q$.  Now define $(p_X^*S)_n$ like $Q_n$ in sequence (\ref{knqn}).  {\it As sheaves of  $\cO_{X \times T}$-modules}, $Q_n$ is isomorphic to $(p_X^*S)_n$ by the following isomorphisms
\begin{equation} \label{tricky}
Q_n \simeq \oplus_{i \leq n} F_i \otimes_\CC \Img(s^{i-1}_D)/\Img(s^i_D) \simeq  \oplus_{i \leq n} S_i \otimes_\CC \Img(s^{i-1}_D)/\Img(s^i_D) \simeq (p_X^*S)_n.
\end{equation}
Taking $n$ sufficiently large, we see that $Q=Q_n \simeq (p_X^*S)_n=p_X^*S$.
\end{proof}
\end{proof}

\section{Properness in the case of polystable orthogonal bundles} \label{polystableorth}

Note that when the framed structure is trivial then we simply have quadratic bundles of degree zero. the aim of this section is to prove the following theorem
\begin{thm} The functor that to a $T$-family associates semi-stable family of polystable orthogonal bundles on $X$ is proper.
\end{thm}
\begin{proof} 
The proof is a deepening of the proof of properness of quadratic bundles with frames by taking trivial framed structure and assuming that the rank $\ker(q_t^W)$ is zero for all $t \in T^*$. We see that all the arguments of last section until extension of the quadratic form $q^W_\Sigma$ go through. Here below we resume the proof from this point. Note that $(W_{\Sigma^*},q^W_{\Sigma^*})$ is a semi-stable family of orthogonal bundles.

\subsection{Definition of quotient sheaf $Q$ and elementary properties}
Consider the sequence which defines the sheaf $Q$
\begin{equation}
 W_\Sigma \stackrel{q^W_\Sigma}{\ra} W^*_\Sigma \ra Q \ra 0.
\end{equation}
The map $q^W_\Sigma$ is injective because its kernel is supported on some nilpotent scheme over $Y$ and thus is killed by some power of $\pi$ on the one hand, and is contained in $W_\Sigma$ which is a vector bundle on $\Sigma$ on the other. Hence we have 
\begin{equation}
\label{Q}
0 \ra W_\Sigma \stackrel{q^W_\Sigma}{\ra} W^*_\Sigma \ra Q \ra 0.
\end{equation}
The reduced subscheme underlying the support of $Q$ is contained in $Y$.
Now for every $m \geq 0$, let 
\begin{equation}
Y_m = X \times \Spec(A/\pi^m).
\end{equation}
Let $n$ denote the least $n$ such that support of $Q$ is contained in $Y_n$. If $n=0$, then our theorem is proved, so we shall suppose here under that $n>0$.

\begin{prop} We have a natural isomorphism $\Tor^1_\Sigma(Q,\cO_{Y_n}) \ra Q$.
\end{prop}
\begin{proof}
We have the exact sequence $0 \ra A \stackrel{\pi^n}{\ra} A \ra A/\pi^n \ra 0$ of modules on $T$. This gives the exact sequence
\begin{equation} \label{Y_n}
0 \ra \cO_\Sigma \stackrel{\pi^n}{\ra} \cO_\Sigma \ra \cO_{Y_n} \ra 0.
\end{equation}
The exactness above follows from the flatness of the projection map $X \times T \ra T$, which is flat being obtained by base change from $X \ra \Spec(\CC)$.
Tensoring (\ref{Y_n}) with $Q$, we get
\begin{equation}
\Tor^1_\Sigma(Q,\cO_\Sigma) \ra \Tor^1_\Sigma(Q,\cO_{Y_n}) \ra Q \stackrel{\pi^n}{\ra} Q \otimes_\Sigma \cO_{\Sigma} \ra Q \otimes_\Sigma \cO_{Y_n} \ra 0.
\end{equation}
Now $\Tor^1_\Sigma(Q,\cO_\Sigma)=0$ because $\cO_\Sigma$ is $\Sigma$-flat. Further $Q \otimes_\Sigma \cO_\Sigma \simeq Q$ and $Q \otimes_\Sigma \cO_{Y_n} \simeq Q$ naturally. The map $Q \stackrel{\pi^n}{\ra} Q$ is zero because $\pi^n$ acts as multiplication by zero on $Q$. Hence the natural map $\Tor^1_\Sigma(Q,\cO_{Y_n}) \ra Q $ is an isomorphism.
\end{proof}

\begin{Cor} \label{torqiso} We have a natural isomorphism $\Tor^1_\Sigma(Q,\cO_Y) \ra Q_Y$.
\end{Cor}
\begin{proof} The proof as above follows by considering the sequence $0 \ra \cO_\Sigma \stackrel{\pi}{\ra} \cO_\Sigma \ra \cO_Y \ra 0$.
\end{proof}

\begin{Cor} \label{notorsion} The sheaf of ideals $Ann(Q)$ associated to the coherent sheaf $Q$ over $ Y_n$ is generated by the image of $\pi$ in $\cO_\Sigma$. The coherent sheaf $Q_Y= Q \otimes_{Y_n} \cO_Y$ has no torsion. \end{Cor}\begin{proof} It is immediate that $Ann(Q)$ is contained in the ideal sheaf generated by $\pi$ in $\cO_\Sigma$. We check the other inclusion as follows. We tensor \ref{Q} with the $\cO_\Sigma$-module $\cO_{Y_n}$ to get \begin{equation} \cdots \ra \Tor^1_\Sigma(W^*_\Sigma,\cO_{Y_n}) \ra \Tor^1_\Sigma(Q,\cO_{Y_n}) \ra W_{Y_n} \ra W^*_{Y_n} \ra Q \ra 0.\end{equation} Now $\Tor^1_\Sigma(W^*_\Sigma,\cO_{Y_n})=0$ because $W^*_\Sigma$ being a family of vector bundles over $\Sigma$ is $\Sigma$-flat. We get \begin{equation} \label{onYn}0 \ra \Tor^1_\Sigma(Q,\cO_{Y_n}) \ra W_{Y_n} \ra W^*_{Y_n} \ra Q \ra 0.\end{equation}We have $Q \simeq \Tor^1_\Sigma(Q,\cO_{Y_n}) \hra W_{Y_n}$. Now we have the inclusions $$(\pi) \subset Ann(W_{Y_n}) \subset Ann(Q).$$

For the second statement, we take $n=1$ above to get

\begin{equation} \label{onY}
0 \ra \Tor^1_\Sigma(Q,\cO_Y) \ra W_Y \ra W^*_Y \ra Q_Y \ra 0.
\end{equation}
This shows that $Q_Y \simeq \Tor^1 \hra W_Y$ is torsion-free.
\end{proof}

\begin{prop} \label{duality} The natual morphism $\Tor^1_\Sigma(Q,\cO_Y) \ra \Hom_Y(Q_Y,\cO_Y)$ is an isomorphism.
\end{prop}
\begin{proof}
We break the sequence (\ref{onY}) into two sequences of coherent $\cO_Y$-modules
\begin{eqnarray}
 0 \ra K \ra W^*_Y \ra Q_Y \ra 0 \\
 0 \ra \Tor^1 \ra W_Y \ra K \ra 0.
\end{eqnarray}
Applying $\Hom_Y(?,\cO_Y)$ we get 
\begin{equation}
0 \ra Q^*_Y \ra W_Y \ra K^* \ra \Ext^1_Y(Q_Y,\cO_Y) \ra 0.
\end{equation}
Now $Q_Y$ being torsion-free on $Y$, which is a smooth projective curve, must be a vector bundle. Thus $\Ext^1_Y(Q_Y,\cO_Y)=0$. Similarly $K \hra W^*_Y$ is torsion-free, which on dualizing gives us the surjective map $W_Y \ra K^*$. Thus we get sequences
\begin{eqnarray}
0 \ra Q^*_Y \ra W_{Y} \ra K^* \ra 0 \\
0 \ra K^* \ra W^*_Y \ra Q^* \ra 0
\end{eqnarray}
Composing them, we see that we get the sequence 
\begin{equation} \label{compdual}
0 \ra Q^*_Y \ra W_Y \stackrel{(q^{W}_Y)^*}{\lra}  W_Y^* \ra \Tor^* \ra 0.
\end{equation}
Since the morphism $q^{W*}_Y=q^W_Y$ is symmetric, hence the sequences (\ref{onY}) and (\ref{compdual}) identify with eachother. So we get an isomorphism (say by $5$-Lemma)
\begin{equation}
\Tor^1_\Sigma(Q,\cO_Y) \ra Q^*=\Hom_Y(Q_Y,\cO_Y).
\end{equation}

\end{proof}

\begin{Cor} \label{degQ_Y0} The sheaf $Q_Y$ is a  vector bundle on $Y$ of degree zero.
\end{Cor} 
\begin{proof} Since $Q_Y$ has no torsion, hence it is a vector bundle on the curve $Y$. The degree computation follows immediately from the fact that $Q_Y$ is isomorphic to $\Tor^1_\Sigma(Q,\cO_Y)$ by Corollary \ref{torqiso} firstly and secondly $\Tor^1_\Sigma(Q,\cO_Y)$  is also its dual by Prop \ref{duality}. 

%Since $W^*_Y \ra Q_Y$ is surjective in (\ref{Q}) and both are of degree zero, hence the semi-stability of $W^*_Y$ implies that $Q_Y$ is semi-stable of degree zero too.
\end{proof}

\begin{prop} We have a natural isomorphism between $\Ext^1_\Sigma(Q,\cO_\Sigma) $ and $Q$.
\end{prop} 
\begin{proof} We dualize (\ref{Q}), by applying $\Hom_\Sigma(?,\cO_\Sigma)$ to it. Since $Q$ is torsion (killed by a power of $\pi$ greater than $n$), so we obtain
\begin{equation}
0= \Hom_\Sigma(Q,\cO_\Sigma) \ra \Hom_\Sigma(W^*_\Sigma, \cO_\Sigma) \stackrel{(q^W_\Sigma)^*}{\lra} \Hom_\Sigma(W_\Sigma,\cO_\Sigma) \ra \Ext^1_\Sigma(Q,\cO_\Sigma) \ra 0.
\end{equation}
Now the dual map $(q^W_\Sigma)^*=q^W_\Sigma$, and so the claimed isomorphism follows.
\end{proof} 

\subsection{Definition of "half" of $Q$}
By going to a cover of $A$ we may suppose that $n$ is even, say $n=2k$.
\begin{prop} \label{I=K} Consider the multiplication $\pi^k: Q \ra Q$. We have \begin{equation}
\Img(\pi^k)=\ker(\pi^k).
\end{equation}
\end{prop}
\begin{proof}
We have a resolution of $A/\pi^k$ as a $A/\pi^n$-module by
\begin{equation}
\cdots \ra A/\pi^n \stackrel{\pi^k}{\ra} A/\pi^n \stackrel{\pi^k}{\ra} A/\pi^n \ra A/\pi^k \ra 0.
\end{equation}
Since the projection map $X \times \Spec(A/\pi^n) = Y_n \ra \Spec(A/\pi^n)$ is flat, so tensoring with $\cO_{Y_n}$ we get a resolution
\begin{equation}
\cdots \ra \cO_{Y_n} \stackrel{\pi^k}{\ra} \cO_{Y_n} \stackrel{\pi^k}{\ra} \cO_{Y_n} \ra  \cO_{Y_k} \ra 0.
\end{equation}
This furnishes the short exact sequence
\begin{equation}
0 \ra \pi^k \cO_{Y_n} \ra \cO_{Y_n} \ra \cO_{Y_k} \ra 0.
\end{equation}
Tensoring with $Q$ over $Y_n$ we obtain
\begin{equation}
\cdots \ra \Tor^1_{Y_n}(Q,\cO_{Y_n}) \ra \Tor^1_{Y_n}(Q,\cO_{Y_k}) \ra \pi^k \cO_{Y_n} \otimes_{Y_n} Q \ra Q \ra Q \otimes \cO_{Y_k} \ra 0.
\end{equation}

\begin{lem}  The sequence \begin{equation}
0 \ra \pi^k \cO_{Y_n} \otimes_{Y_n} Q \ra Q \ra Q \otimes_{Y_n} \cO_{Y_k} \ra 0.
\end{equation} is exact.
\end{lem}
\begin{proof}
Now $ \Tor^i_{Y_n}(Q,\cO_{Y_n})=0$ for $i \geq 1$ since $\cO_{Y_n}$ is flat. This implies that we have an exact sequence \begin{equation} \label{tobeshownzero}
0 \ra \Tor^1_{Y_n}(Q,\cO_{Y_k}) \ra \pi^k \cO_{Y_n} \otimes_{Y_n} Q \ra Q \ra Q \otimes \cO_{Y_k} \ra 0.
\end{equation} 
and that the natural morphism 
\begin{equation} \Tor^{i+1}_{Y_n}(Q,\cO_{Y_n}) \ra \Tor^i_{Y_n}(Q,\cO_{Y_k})
\end{equation} is an isomorphism for $i \geq 1$. On the other hand, viewing the sequence $0 \ra W_\Sigma \ra W^*_\Sigma \ra Q \ra 0$ as a projective $\cO_\Sigma$-resolution of $Q$, we tensor it by the $\cO_\Sigma$-module $\cO_{Y_k}$ to calculate the higher $\Tor_\Sigma$
\begin{equation}
\cdots \ra 0 \ra W_\Sigma \otimes_\Sigma \cO_{Y_k} \ra W^*_\Sigma \otimes_\Sigma \cO_{Y_k} \ra Q \otimes_\Sigma \cO_{Y_k} \ra 0.
\end{equation}
This shows that $\Tor_\Sigma^i(Q, \cO_{Y_k})=0$ for $i \geq 2$. But we remark since $\cO_{Y_k} \otimes_\Sigma \cO_{Y_n} = \cO_{Y_k}$, so it follows that for all $i$ we have
\begin{equation}
\Tor_\Sigma^i(Q, \cO_{Y_k}) = \Tor_{Y_n}^i(Q,\cO_{Y_k}).
\end{equation}

This means that $\Tor_{Y_n}^i(Q,\cO_{Y_k})= \Tor_{Y_n}^{i+1}(Q,\cO_{Y_k})= \Tor_{S}^{i+1}(Q,\cO_{Y_k})=0$ also for all $i \geq 1$. In particular,  
$\Tor_{Y_n}^1(Q,\cO_{Y_k})=0$, which reduces the sequence (\ref{tobeshownzero}) to the exact sequence
\begin{equation} \label{inj}
0 \ra \pi^k \cO_{Y_n} \otimes_{Y_n} Q \ra Q \ra Q \otimes_{Y_n} \cO_{Y_k} \ra 0.
\end{equation}

\end{proof}

 We tensor
\begin{equation}
\xymatrix{
\cO_{Y_n} \ar[rr]^{\pi^k} \ar@{>>}[rd] & & \cO_{Y_n} \\
& \cO_{Y_k} \simeq  \pi^k\cO_{Y_n}  \ar@{^{(}->}[ru]
}
\end{equation}
with $Q$, and remark that the map $Q \otimes_{Y_n} \pi^k\cO_{Y_n} \ra Q$ is already shown to be injective in (\ref{inj}), to get
\begin{equation}
\xymatrix{
Q \ar[rr]^{\pi^k} \ar@{>>}[rd] & & Q \\
& Q \otimes_{Y_n} \cO_{Y_k} \simeq Q \otimes_{Y_n} \pi^k\cO_{Y_n} \ar@{^{(}->}[ru]
}.
\end{equation}
This has two consequences
\begin{enumerate}
\item 
 $\ker(\pi^k) =\ker(Q \ra Q \otimes_{Y_n} \cO_{Y_k})$. The latter by (\ref{inj}) equals $ \pi^k \cO_{Y_n} \otimes_{Y_n} Q$.
 \item $\Img(\pi^k)$ in $Q$ equals $ \pi^k \cO_{Y_n} \otimes_{Y_n} Q$ 
 \end{enumerate}
Thus $\Img(\pi^k)=\Ker(\pi^k)$.

\end{proof}

Using Prop \ref{I=K}, we can define the ``half'' of $Q$ as follows. Consider the image of $mult(\pi^k): Q \ra Q$. It is $\pi^kQ$ on the one hand. On the other, it is $Q/\ker(\pi^k)$ which by Prop \ref{I=K} is equal to $Q/\Img(\pi^k)= Q/\pi^kQ$. Thus we have an isomorphism
\begin{equation}
\pi^kQ \simeq Q/\pi^kQ.
\end{equation}

\begin{Cor} We have  $\deg(Q_{1/2Y})=0$.
\end{Cor} 
\begin{proof} This simply follows from the isomorphism $Q_Y \ra Q_{1/2Y}$ and the degree of $Q_Y$.
\end{proof}
\begin{Cor} \label{connhom} The connecting homomorphism $\Tor^1_\Sigma(Q_{1/2},\cO_Y) \ra Q_{1/2Y}$ is an isomorphism.
\end{Cor}
\begin{proof} We tensor $0 \ra Q_{1/2} \ra Q \ra Q_{1/2} \ra 0$ by $\cO_Y$ to get 
\begin{equation}
\Tor^1_\Sigma(Q,\cO_Y) \ra \Tor^1_\Sigma(Q_{1/2},\cO_Y) \ra Q_{1/2Y} \ra Q_Y \ra Q_{1/2Y} \ra 0.
\end{equation}
Now the map $Q_Y \ra Q_{1/2Y}$ is an isomorphism. So it suffices to show that the morphism $\Tor^1_\Sigma(Q,\cO_Y) \ra \Tor^1_\Sigma(Q_{1/2},\cO_Y)$ is zero. This map is obtained by $mult(\pi^k): Q \ra Q$, which induces $mult(\pi^k): \Tor^1_\Sigma(Q,\cO_Y) \ra \Tor^1_\Sigma(Q,\cO_Y)$. Now the $A$-module structure on $\Tor^1_\Sigma$ can be obtained also by multiplication on the second factor $\cO_Y$ on which $\pi^k$ acts as zero.

\end{proof}
\begin{defi} We define ``half'' $Q_{1/2}$ of $Q$ as $\pi^kQ \simeq Q/\pi^kQ \hra Q$.
\end{defi} 
Sometimes we may suggestively write
\begin{equation}
0 \ra Q_{1/2} \ra Q \ra Q^{1/2} \ra 0.
\end{equation}

The following proposition is an ingredient in main Prop \ref{keyfactorization} used to check naturality.
\begin{prop} \label{halfext} The sub-module $\Ext^1_\Sigma(Q^{1/2},\cO_\Sigma) \ra \Ext^1_\Sigma(Q,\cO_\Sigma)$ is naturally the half of the latter. Also under the natural isomorphism $\Ext^1_\Sigma(Q,\cO_\Sigma) \simeq Q$ it identifies with $Q_{1/2}$.
\end{prop}
\begin{proof}
Dualizing the diagram 
\begin{equation}
\xymatrix{
Q \ar[r]^{\pi^k} \ar@{>>}[d] &              Q \ar[r] & Q^{1/2} \ar[r] & 0 \\
                   Q^{1/2} \simeq Q_{1/2} \ar@{^{(}->}[ru] & &&
}
\end{equation}
by applying $\Hom_\Sigma(?,\cO_\Sigma)$  we obtain
\begin{equation*}
\xymatrix{
 \Ext_\Sigma^1(Q^{1/2},\cO_\Sigma) \ar@{^{(}->}[r] &  \Ext^1_\Sigma(Q,\cO_\Sigma) \ar[r]^{\pi^k} \ar@{>>}[d] &             \Ext^1_\Sigma(Q,\cO_\Sigma) \\
 & \Ext^1_\Sigma( Q_{1/2},\cO_\Sigma) \simeq \Ext^1_\Sigma(Q^{1/2},\cO_\Sigma) \ar@{^{(}->}[ru]  &
}
\end{equation*}

We make two remarks here. Firstly, note  that dualizing $mult(\pi^k): Q \ra Q$ we  again get multiplication by $\pi^k$ as the induced endomorphism map on $ \Ext^1_\Sigma(Q,\cO_\Sigma)$  because the $A$-module structure on $\Ext$ can be defined by the first (or equivalenty the second) factor. Secondly, note that $\Ext^2_\Sigma(Q,\cO_\Sigma)=0$ because by the sequence (\ref{Q}) provides a projective resolution of $Q$ of length two. Thus the downward vertical arrow is surjective.

Thus $\Ext^1_\Sigma(Q^{1/2},\cO_Y)$ being naturally the image and kernel of $mult(\pi^k)$ identifies with the half of $\Ext^1_\Sigma(Q,\cO_Y)$ and under the isomorphism $\Ext^1_\Sigma(Q,\cO_Y)$ it identifies with $Q_{1/2}$.
\end{proof}

Summarizing we have
\begin{equation} \label{Q/2seq}
\xymatrix{
0 \ar[r] & Q_{1/2} \ar[r] \ar[d] & Q \ar[r] \ar[d] & Q^{1/2} \ar[r] \ar[d] & 0 \\
0 \ar[r] & \Ext^1_\Sigma(Q^{1/2},\cO_Y) \ar[r] & \Ext^1_\Sigma(Q,\cO_Y) \ar[r] & \Ext^1_\Sigma(Q_{1/2},\cO_Y) \ar[r] & 0
}
\end{equation}

\subsection{Key Hecke-modification}
By taking pull-out of short exact sequence $0 \ra W_\Sigma \stackrel{q^W_\Sigma}{\ra} W^*_\Sigma \ra Q \ra 0$ by $Q_{1/2} \hra Q$, we obtain
\begin{equation} \label{extW1}
0 \ra W_\Sigma \ra U_\Sigma \ra Q_{1/2} \ra 0.
\end{equation}

A reader will observe that the proof of the following proposition is very similar to that of Proposition \ref{fact-gen-non-deg}.
\begin{prop} \label{keyfactorization} We have a factorization
\begin{equation} 
\xymatrix{
W_\Sigma \ar[r]^{q^W_\Sigma} \ar[d] & W^*_\Sigma \\
U_\Sigma \ar@{.>}[r]^{q^U_\Sigma} & U_\Sigma^* \ar[u]
}
\end{equation}
where the map $q^U_\Sigma : U_\Sigma \ra U^*_\Sigma$ is a symmetric isomorphism.
\end{prop}
\begin{proof} Dualizing of the extension class (\ref{extW1}), we get  
\begin{equation} \label{dualextW1}
0 \ra U_\Sigma^* \ra W_\Sigma^*\ra \Ext_\Sigma^1(Q_{1/2}, \cO_Y) \ra  0
\end{equation}
We shall prove that the composite of $W_\Sigma \stackrel{q^W_\Sigma}{\lra} W_\Sigma^*$ with $W_\Sigma^* \ra \Ext_\Sigma^1(Q_{1/2}, \cO_\Sigma)$ is zero, so by (\ref{dualextW1}), we would have a factorization 
\begin{equation} \label{start2}
\xymatrix{
       & W_\Sigma \ar[d]^{q^W_\Sigma} \ar@{.>}[ld]^{q_1} & \\
U_\Sigma^* \ar[r] & W_\Sigma^* \ar[r] & \Ext_\Sigma^1(Q_{1/2}, \cO_\Sigma)
}
\end{equation} 
Then we will show that composing  $q_1^*: U_\Sigma \ra W_\Sigma^*$ with $W_\Sigma^* \ra \Ext_\Sigma^1(Q_{1/2}, \cO_\Sigma)$ is zero. This furnishes the desired factorization
\begin{equation*}
\xymatrix{
      & U_\Sigma \ar[d]^{q_1^*} \ar@{.>}[ld]_{q^U_\Sigma} & \\
U_\Sigma^* \ar[r] & W_\Sigma^* \ar[r] & \Ext_\Sigma^1(Q_{1/2}, \cO_\Sigma).
}
\end{equation*}
Let us remark that $W_\Sigma^* \ra \Ext_\Sigma^1(Q_{1/2}, \cO_\Sigma)$ corresponds to evaluating the sequence (\ref{extW1}).

Firstly, since the extension (\ref{Q}) $\in \Ext^1_\Sigma(Q,W_\Sigma)$ arises from  $\Id \in \Hom_\Sigma(Q, Q)$ under the connecting homomorphism, so in $\Ext_\Sigma^1(Q, W_\Sigma^*)$ its image is zero. This corresponds to taking the push-out of (\ref{Q}) by $q^W_\Sigma: W_\Sigma \ra W_\Sigma^*$. It follows now by the commuting squares,
 \begin{equation*}
\xymatrix{
\ar[r] &  \Hom_\Sigma(Q, Q) \ar[r] \ar[d] & \Ext_\Sigma^1(Q, W_\Sigma) \ar[r] \ar[d] & \Ext_\Sigma^1(Q, W_\Sigma^*) \ar[d] \\
\ar[r] &  \Hom_\Sigma(Q_{1/2}, Q) \ar[r]      & \Ext_\Sigma^1(Q_{1/2}, W_\Sigma) \ar[r]      & \Ext_\Sigma^1(Q_{1/2}, W_\Sigma^*) 
}
\end{equation*}
the  push-out of (\ref{extW1}) by $q^W_\Sigma: W_\Sigma \ra W_\Sigma^*$ becomes the trivial extension
\begin{equation} \label{text2}
\xymatrix{
0 \ar[r] & W_\Sigma \ar[r] \ar[d]^{p^*q} & U_\Sigma \ar[r] \ar@{.>}[d] & Q_{1/2} \ar[r] \ar@{.>}[d] & 0 \\
0 \ar[r] & W_\Sigma^* \ar[r]                          & W_\Sigma^* \oplus Q_{1/2} \ar[r] & Q_{1/2} \ar[r] & 0.
}
\end{equation}
The composite of $W_\Sigma \stackrel{q^W_\Sigma}{\ra} W_\Sigma^* \ra \Ext_\Sigma^1(Q_{1/2}, \cO_\Sigma)$, corresponds to evaluating by  $w \in W_\Sigma$ to get push out of  the bottom row of (\ref{text2}). Taking the push-out of (\ref{text2}) which is the trivial extension,
\begin{equation}
\xymatrix{
0 \ar[r] & W_\Sigma \ar[r]  \ar@{->}[d]^{\ev(w)}            & W_\Sigma \oplus Q_{1/2} \ar[r] \ar@{.>}[d]^{\ev(w)} & Q_{1/2} \ar[r] \ar@{.>}[d]& 0 \\
0 \ar[r] & \cO_Y \ar[r] & \cO_Y \oplus Q_{1/2} \ar[r] & Q_{1/2} \ar[r] & 0 
}
\end{equation}
we again get trivial extensions. This means, in other words, that the composite is zero. This furnishes the map $q_1 : W_\Sigma \ra U_\Sigma^*$.  

By the sequence (\ref{dualextW1}) and (\ref{Q}), we have a factorization in equation (\ref{start2})
\begin{equation} 
\xymatrix{
       & W_\Sigma \ar[d]^{q^W_\Sigma} \ar[ld]^{q_1} & \\
U_\Sigma^* \ar[r] & W_\Sigma^* \ar[r] \ar[d] & \Ext_\Sigma^1(Q_{1/2}, \cO_\Sigma) \\
            & Q \ar@{.>}[ru] &
}
\end{equation} 

where $Q \ra \Ext_\Sigma^1(Q_{1/2}, \cO_\Sigma)$ is a natural surjective map. By Prop \ref{halfext}, we have $\Ext_\Sigma^1(Q_{1/2}, \cO_\Sigma) \simeq Q_{1/2}$ and thus  $\ker( Q \ra \Ext_\Sigma^1(Q_{1/2}, \cO_\Sigma)) $ identifies with $Q_{1/2}$ by sequence (\ref{Q/2seq}). We thus get
$$0 \ra W_\Sigma \stackrel{q_1}{\ra} U_\Sigma^* \ra Q_{1/2} \ra 0$$
by the diagram (\ref{start2}).
Taking duals, we obtain
$$0 \ra U_\Sigma \stackrel{q_1^*}{\ra}  W_\Sigma^* \ra \Ext_\Sigma^1(Q_{1/2}, \cO_\Sigma) \ra 0.$$
Thus the composite of $q_1^*$ with $W_\Sigma^* \ra \Ext_\Sigma^1(Q_{1/2}, \cO_\Sigma)$ is zero. So we obtain a map $q^U_\Sigma: U_\Sigma \ra U_\Sigma^*$ factoring $q^W_\Sigma: W_\Sigma \ra W_\Sigma^*$. 

By the sequence (\ref{extW1}), and because $Q_{1/2}$ is the "half" of $Q$, it follows that this factoring is an isomorphism.
\end{proof}

\begin{prop} The family $U_\Sigma$ is a semi-stable family of orthogonal bundles.
\end{prop} 
\begin{proof} Recall that the extension (\ref{extW1}) was obtained from pulling out by $Q_{1/2} \hra Q$ as follows
\begin{equation}
\xymatrix{
0 \ar[r] & W_\Sigma \ar[r] & W^*_\Sigma \ar[r] & Q \ar[r] & 0 \\
0 \ar[r] & W_\Sigma \ar[r] \ar[u] & U_\Sigma \ar[r] \ar[u] & Q_{1/2} \ar[r] \ar[u] & 0 
}.
\end{equation}
We remark that the induced endomorphism of  $W_\Sigma$ is identity.
Tensoring with the $\cO_\Sigma$-module $\cO_Y$ we get
\begin{equation} \label{compdiag}
\xymatrix{
& 0 \ar[r] & Q^{1/2}_Y \ar@{>>}[r]  & Q^{1/2}_Y    \\
 \Tor_1^\Sigma(Q,\cO_Y) \ar@{{(}^->}[r] & W_Y \ar[r] \ar[u] & W_Y^* \ar@{>>}[u] \ar@{>>}[r] & Q_Y  \ar@{>>}[u]  \\
 \Tor_1^\Sigma(Q,\cO_Y) \ar@{{(}^->}[r] \ar[u]^{\Id} & W_Y \ar[r] \ar[u] & U_Y \ar[u] \ar@{>>}[r] & Q_{1/2Y}  \ar@{.>}[u]^{zero} \\
&0 \ar[r] \ar[u] & \Tor^{\Sigma}_1(Q^{1/2},\cO_Y) \ar@{{(}^->}[u] \ar@{>>}[r]^{\Id} & \Tor^{\Sigma}_1(Q^{1/2},\cO_Y) \ar@{{(}^->}[u]^{\simeq}  
}
\end{equation}
where the map $Q_{1/2Y} \ra Q_Y$ is seen to be zero. It follows that the vertical map $\Tor^\Sigma_1(Q^{1/2},\cO_Y) \ra Q_{1/2Y}$ is an isomorphism.

From the above diagram it follows that the sequence $U_Y \ra Q_{1/2Y}$ splits. Now $Q_Y$ is semi-stable of degree zero. This follows from \ref{Q}, since $W^*_Y$ is semi-stable of degree zero and by Prop \ref{degQ_Y0}, the degree of $Q_Y$ is zero. Similarly $\ker(U_Y \ra Q_{1/2Y}) = W_Y/\Tor^\Sigma_1$ is semi-stable of degree zero being quotient of $W_Y$ and since $\deg(\Tor^\Sigma_1)$ is also zero being isomorphic to $Q_Y$ by Corollary \ref{torqiso}. Thus $U_Y$ is semi-stable of degree zero. That $U_Y$ is $S$-equivalent to $W_Y$ follows from \ref{compdiag}.

\end{proof}

The $S$-equivalence class of the family $(U_\Sigma ,q^U_\Sigma)$ is the sought equivalence class. To see this, note that our construction depends on the limit $W_Y$ we chose of $W_{\Sigma^*}$. Had we had chosen another bundle $W'_Y$ which is only $S$-equivalent to it, we would have got a new orthogonal bundle $U'_Y$ which is $S$-equivalent to $W'_Y$. But since $W'_Y$ and $W_Y$ are $S$-equivalent by assumption, hence it follows that so are $U'_Y$ and $U_Y$ as vector bundles. In particular, if the limit $W_Y$ is polystable then the new orthogonal bundle is also polystable as a vector bundle. Now one checks that polystable vector bundle together with orthogonal structure are polystable as orthogonal bundles too. 

\end{proof}

\begin{rem}  \label{diff} After the diagram \ref{compdiag}, we mentioned that $U_Y \ra Q_Y$ is split. But the sequence $\Tor_1^\Sigma(Q,\cO_Y) \hra W_Y$ may not split. So when $\Shalf$ is non-trivial, it does not seem possible to transport the framed structure from $W_Y$ to $U_Y$. On the other hand,  the quadratic form $q^W_\Sigma$ becomes non-degenerate only on $U_\Sigma$. So for framed quadratic modules, it seems inevitable to obtain everywhere degenerate quadratic bundles (with frames) in the limit at infinity.
\end{rem}

\section{Comparisons with \cite{gs} and \cite{lgp}}

In this section, we compare our (semi)-stability conditions with that of \cite[Gomez-Sols]{gs} and \cite{lgp} for $(V,q)$ where $q$ is generically non-degenerate.

%\begin{prop} \label{degupdown} Let $p: Y \ra X$ be a double cover of curves. Put $d = \deg(p_* \cO_Y)$. For a vector bundle $V$ on $X$, we have $2 \mu (V) + 3/2 d = \mu(p^*(V))$. \end{prop}

Note our definition of semi-stability for quadratic framed modules $(f:W \ra \Shalf,q')$, for their invariant direct image $(V,q)$ simply mean that the underlying vector bundle $V$ is (semi)-stable.

Let us recall the definition of (semi)-stability of conic bundles. 

\begin{defi} Let $(\cE,Q)$ be a quadric bundle. We say that two sub-bundles $\cE_1 \subset   
\cE_2 \subset \cE$ give a critical filtration of $(\cE,Q)$ if $\rank(\cE_1)=1$, $\rank(\cE_2)=2$, $Q|_{\cE_1 \cE_2}=0$, and $Q|_{\cE_1 \cE} \neq 0 \neq Q|_{\cE_2 \cE}$.
\end{defi}

\begin{defi} Let $\tau$ be a positive rational number and $(E,Q)$ be a conic bundle. We say that $(E,Q)$ is (semi)-stable with respect to $\tau$ if the following conditions hold
\begin{enumerate}
\item (ss.1) If $E'$ is a proper sub-bundle of $E$, then
$$\frac{\deg(E') - c_Q(E') \tau}{\rank(E')} \leq \frac{\deg(E) - 2 \tau}{3}$$
where $c_Q(E')$ equals $2$ if $Q(E',E') \neq 0$, else $1$ if $Q(E',E) \neq 0$ and $0$ otherwise,
\item (ss.2) If $E_1 \subset E_2 \subset E$ is a critical filtration, then 
$\deg(E_1) + \deg(E_2) \leq \deg(E).$
\end{enumerate}

\end{defi}

\begin{defi} \cite[Definition 16,page 10]{lgp} The above definition is generalized to any rank by retaining (ss1) and replacing the second condition with the following.
For every critical filtraion $0 \subset E_i \subset E_j \subset E$ we have
$$ (r_i + r_j)d - r (d_i + d_j) - 2 \tau (r_i + r_j -r) (\geq) 0.$$
\end{defi}
\begin{prop} For $\delta =1$ and $(V,q)$ where $q$ is generically non-degenerate, the $\delta$-(semi)-stability condition implies [GS] semi-stability condition for all $\tau > 0$. \end{prop}
\begin{proof}
Let us check ss1. Let $V'$ be a sub-bundle of $V$. Now $$\frac{\deg(V) - 2 \tau}{3} - \frac{\deg(V') - c_q(V') \tau}{\rank(V')} = \frac{\mu(V) - \mu(V')}{2} + \tau (\frac{c_q(V')}{\rank(V')} - 2/3).$$ Now for $(c_q(V')=1, \rank(V')=2)$ and $(c_q(V')=0, \rank(V')=1,2)$, the bundle $V$ would not be non-degenerate. In the remaining cases, the expression is positive.

Let us check ss2. Note that for conic bundles, i.e when $\rank(E)=3$, if $0 \subset E_1 \subset E_2 \subset E$ forms a critical filtration, then $E_2 = E_1^\perp$.
 Then by our semi-stability conditions, we have 
\begin{eqnarray}
\frac{\deg(V') }{1} \leq \frac{\deg(V) }{3} \\
\frac{\deg(V'^\perp) }{2} \leq \frac{\deg(V) }{3}.
\end{eqnarray}

These imply the condition ss2. 
\end{proof}

\begin{prop} For $\tau < 1/16$, stability in [GS] implies $\delta$-(semi)-stability condition when $q$ is generically non-degenerate.
\end{prop}
\begin{proof}
Let $W'$ be a sub-bundle of $W$. Then if $\tau < 1/16$, then [GS]-stability implies that $\mu(W) - \mu(W') + \epsilon > 0 $ where $\epsilon < 1/6$. This means that $\mu(W) - \mu(W') \geq 0$ because the least common multiple of the possible ranks from $1$ to $3$ is $6$.
\end{proof}

\begin{rem} For the sake of completeness we justify that the parameter $\tau$ in Gomez-Sols \cite{gs}, can be taken to be as small or as large as we like. It is defined in Prop 2.10 of \cite[ GS]{gs} as 
$$\frac{n_2}{n_1} = \frac{P(l) - P(m)}{P(m) - 2 \tau} \tau$$
where $n_1$ and $n_2$ are the multiples of the principal polarizations on $\Quot(H \otimes \cO_X(-m), P)$ and $\Proj(\Sym^2(H^* \otimes H^0(X, \cO_X(l-m))))$, $P$ is the Hilbert polynomial of the underlying vector bundles ($P(n)= 3n + d + 3(1-g)$) and $H$ is a vector space of dimension $P(m)$. For fixed $m, l$ and $n_2$ as $n_1$ tends to $\infty$, $\tau$ tends to $0$. So we can make it arbitrarily small. Keeping the difference $n-m$, $n_1$ and $n_2$ fixed, by letting $m$ tend to $\infty$, we can get $\tau$ arbitrarily large because $\tau= 2n_2 \frac{ 3m + d + 3(1-g)}{3 n_1 ((n-m) + 2n_2)}$.
\end{rem}

We use the following facts to check (ss1) which is always satisfied when $q$ is generically non-degenerate. Note here that since $q$ is generically non-degenerate, so for any sub-bundle $V'$, $c_q(V') \neq 0$. Also if $\rank(V') > \rank(V)/2$, then $c_q(V') \geq 1$. It is for the second condition that one needs to take $\tau$ small exploiting the fact that since the degree and the rank are integers, thus for stable bundles, the supremum of slopes is still bounded above.

In conclusion, with the methods of the last section, we may check that for a generically non-degenerate quadratic bundle $(V,q)$ which is $\delta=1$ stable as per our definition is $\tau$-stable in the sense of \cite{lgp} for small values of $\tau$.

\section*{Acknowledgements}
This work started during my post-doc at Chennai Mathematical Institute. Prof. C.S.Seshadri suggested  this problem and encouraged. Prof. V.Balaji constantly supported and generously explained the non-GIT approach of \cite{bs}. I thank them for the confidence they have placed in me. My thanks also go to the Institute of Mathematical Sciences, Chennai for providing local hospitality.

\bibliographystyle{amsplain}
\bibliography{qbun}
\end{document}